\documentclass[a4paper,reqno]{amsart}
\usepackage[english]{babel}
\usepackage{amsmath,amsthm,amssymb,amsfonts}

\usepackage{hyperref}

\usepackage{enumitem}
\setlist[enumerate]{leftmargin=*}
\setlist[itemize]{leftmargin=2em}

\newtheorem{thm}{Theorem}[section]
\newtheorem{lem}[thm]{Lemma}
\newtheorem{prp}[thm]{Proposition}
\newtheorem{cor}[thm]{Corollary}

\theoremstyle{definition}
\newtheorem{rem}[thm]{Remark}

\numberwithin{equation}{section}

\newcommand{\NN}{\mathbb{N}}
\newcommand{\ZZ}{\mathbb{Z}}
\newcommand{\RR}{\mathbb{R}}
\newcommand{\CC}{\mathbb{C}}

\newcommand{\opL}{\mathcal{L}}
\newcommand{\Rz}{\mathcal{R}}

\newcommand{\Sz}{\mathcal{S}}
\newcommand{\Df}{\mathcal{D}}

\newcommand{\opH}{\mathcal{H}}

\newcommand{\Four}{\mathcal{F}}

\newcommand{\Part}{\mathcal{P}}

\newcommand{\tc}{\,:\,}

\newcommand{\defeq}{\mathrel{:=}}

\DeclareMathOperator{\arccosh}{arccosh}
\DeclareMathOperator*{\esssup}{ess\,sup}

\title{Riesz transforms on $ax+b$ groups}

\author{Alessio Martini}
\address{Dipartimento di Scienze Matematiche ``G. L. Lagrange'' \\ Dipartimento di Eccellenza 2018-2022 \\ Politecnico di Torino \\ Corso Duca Degli Abruzzi 24 \\ 10129 Torino \\ Italy}
\email{alessio.martini@polito.it}

\thanks{The financial support of the Compagnia di San Paolo is gratefully acknowledged. The author is a member of the Gruppo Nazionale di Analisi Matematica, Probabilit\`a e Applicazioni (GNAMPA) of the Istituto Nazionale di Alta Matematica (INdAM)}

\begin{document}
\begin{abstract}
We prove the $L^p$-boundedness for all $p \in (1,\infty)$ of the first-order Riesz transforms
$X_j \opL^{-1/2}$ associated with the Laplacian $\opL = -\sum_{j=0}^n X_j^2$ on the $ax+b$-group $G = \RR^n \rtimes \RR$; here $X_0$ and $X_1,\dots,X_n$ are left-invariant vector fields on $G$ in the directions of the factors $\RR$ and $\RR^n$ respectively. This settles a question left open in previous work of Hebisch and Steger (who proved the result for $p \leq 2$) and of Gaudry and Sj\"ogren (who only considered $n=1=j$).
The main novelty here is that we can treat the case $p \in (2,\infty)$ and include the Riesz transform 
in the direction of $\RR$; an operator-valued Fourier multiplier theorem on $\RR^n$ turns out to be key to this purpose.
We also establish a weak type $(1,1)$ endpoint for the adjoint Riesz transforms in the direction of $\RR^n$.
By transference, our results imply the $L^p$-boundedness for $p \in (1,\infty)$ of the first-order Riesz transforms 
associated with the Schr\"odinger operator $-\partial_s^2 + e^{2s}$ on the real line.
\end{abstract}

\maketitle

\section{Introduction}

Let $G = \RR^n \rtimes \RR$, where $\RR$ acts on $\RR^n$ via dilations. If we write the elements of $G$ as $(x,u)$, where $x \in \RR^n$ and $u \in \RR$, then the group operation is given by
\begin{equation}\label{eq:operation}
(x,u) \cdot (x',u') = (x+ e^u x', u+u').
\end{equation}
The group $G$ is isomorphic to the group of transformations of $\RR^n$ generated by translations and dilations, also known as the $ax+b$ group (see, e.g., \cite[Section 6.7]{F}).
We equip $G$ with the right Haar measure, that is, the Lebesgue measure $dx \, du$. A basis of left-invariant vector fields on $G$ is given by
\begin{equation}\label{eq:vfs}
X_0 = \partial_u, \quad X_1 = e^u \partial_{x_1}, \quad \dots, \quad X_n = e^u \partial_{x_n}.
\end{equation}
The corresponding sum-of-square operator
\[
\opL = -\sum_{j=0}^n X_j^2 = -\partial_u^2 - e^{2u} \Delta_x
\]
is essentially self-adjoint on $L^2(G)$. If we equip $G$ with the natural metric structure associated with $\opL$, namely, the left-invariant Riemannian metric for which the vector fields \eqref{eq:vfs} form an orthonormal frame, then $G$ with the Riemannian distance and the Haar measure is a metric measure space of exponential volume growth.

In this work, we are interested in Lebesgue boundedness properties of the first-order Riesz transforms
\[
\Rz_j = X_j \opL^{-1/2}, \qquad j=0,\dots,n,
\]
associated with the left-invariant Laplacian $\opL$ on $G$. Our first main result reads as follows.

\begin{thm}\label{thm:main}
For $j=0,\dots,n$, the Riesz transform $\Rz_j$ is bounded on $L^p(G)$ for all $p \in (1,\infty)$.
\end{thm}

For $p \leq 2$, the above boundedness result was proved in \cite{HS}. (See also \cite{S} for previous partial results, and \cite{MV} for an extension to the case $G = N \rtimes \RR$, where $N$ is a stratified Lie group.) To this purpose, in \cite{HS} the authors developed a Calder\'on--Zygmund theory adapted to the nondoubling structure of $G$, and showed that the integral kernel of $\Rz_j$ satisfies a suitable ``integral H\"ormander condition'', which implies that $\Rz_j$ is of weak type $(1,1)$ and $L^p$-bounded for $p\in(1,2]$. The same ``integral H\"ormander condition'' also implies that the $\Rz_j$ are bounded from $H^1(G)$ to $L^1(G)$, where $H^1(G)$ is the atomic Hardy space on $G$ introduced in \cite{V}.

As it turns out, the adjoint Riesz transforms $\Rz_j^*$ do not satisfy the aforementioned ``integral H\"ormander condition''. More is true: the operators $\Rz_j^*$ are not bounded from $H^1(G)$ to $L^1(G)$. This fact, 
which was already discovered in \cite{SV} in the case $n=2$,
demonstrates that the approach used for $p \leq 2$ does not immediately extend to $p>2$ by duality considerations, and a different approach is needed.

Regarding the $L^p$-boundedness of the Riesz transforms $\Rz_j$ for $p>2$, to the best of our knowledge, the only result so far available in the literature has been the one contained in \cite{GS} for $n=1$ and $j=1$.
However, the method of \cite{GS} appears not to be suitable to treat the Riesz transform $\Rz_0 = X_0 \opL^{-1/2}$, and indeed the authors leave open the question of its $L^p$-boundedness. One of the main reasons of interest of Theorem \ref{thm:main} is that it also includes the case $j=0$, thus solving a problem that has been open for some time even in the smallest dimensional case $n=1$.

Our proof of Theorem \ref{thm:main} does not yield an endpoint bound for $p=\infty$. In particular, we do not know whether $\Rz_0^*$ is of weak type $(1,1)$. However, we can obtain such an endpoint for the operators $\Rz_j^*$ with $j>0$ by a different method, namely, an extension of the approach of \cite{GS} to the higher-dimensional cases.

\begin{thm}\label{thm:wt11}
For $j=1,\dots,n$, the adjoint Riesz transform $\Rz_j^*$ is of weak type $(1,1)$.
\end{thm}

The $L^p$-boundedness of Riesz transforms associated with elliptic and sub-elliptic operators on manifolds is a widely studied problem (see, e.g., \cite{Str,Ba,AL,LM,CD2,Si,ACDH,Li}), and it is impossible to include here a complete literature review on the subject. We refer to the introduction of \cite{MV} for a discussion of those works which are most closely related to our setting. The fact that the $L^p$-boundedness of Riesz transforms for $p>2$ is a more delicate property than that for $p \leq 2$ has already been observed in other contexts, see, e.g., \cite{CD1,ACDH}; an additional difficulty in our case is that here we are working in a nondoubling setting.

We remark that the Laplacian $\opL$ on $G$ considered here is not the Laplace--Beltrami operator for the aforementioned left-invariant Riemannian structure on $G$.
Indeed, the group $G$ with such Riemannian structure is nothing else than a realisation of the real hyperbolic space of dimension $n+1$, and the corresponding Laplace--Beltrami operator $\opL_H$ (which is self-adjoint with respect to the \emph{left} Haar measure)
is given by
\[
\opL_H = \opL + n X_0,
\]
i.e., $\opL$ and $\opL_H$ differ by a drift term. In these respects, the operator $\opL$ can be thought of as the natural Laplacian on the \emph{weighted} Riemannian manifold $G$, where the weight is the modular function (i.e., the density of the right Haar measure with respect to the left Haar measure). Most importantly, and differently from $\opL_H$, the Laplacian $\opL$ has no spectral gap, so there is no ``extra decay'' of the associated heat propagator for large time that can be exploited to deal with the part at infinity of the Riesz transforms associated with $\opL$. In addition, here we work with the ``unshifted'' Riesz transforms $\Rz_j = X_j \opL^{-1/2}$, as opposed to the ``shifted'' Riesz transforms $X_j (a+\opL)^{-1/2}$ for some $a>0$ considered elsewhere in the literature. Correspondingly, the convolution kernels of the $\Rz_j$ are not integrable at infinity, which, combined with the exponential volume growth of $G$, makes their analysis a nontrivial problem.

The group $G$ has a family of irreducible infinite-dimensional unitary representations on $L^2(\RR)$, in each of which the distinguished Laplacian $\opL$ on $G$ corresponds to the Schr\"odinger operator $\opH = -\partial_s^2 + e^{2s}$ on the real line. As $G$ is an amenable group, the $L^p$-boundedness of the Riesz transforms associated with $\opL$ transfers to an analogous result for the Schr\"odinger operator $\opH$, which may be of independent interest.

\begin{thm}\label{thm:schroedinger_riesz}
The first-order Riesz transforms
\[
\partial_s \opH^{-1/2}, \qquad V^{1/2} \opH^{-1/2}
\]
associated with the Schr\"odinger operator $\opH = -\partial_s^2 + V(s)$, where $V(s) = e^{2s}$, are bounded on $L^p(\RR)$ for all $p \in (1,\infty)$.
\end{thm}

The $L^p$-boundedness of Riesz transforms associated with Schr\"odinger operators has also been studied extensively in the literature, and indeed the $L^p$ bounds for $p \leq 2$ in Theorem \ref{thm:schroedinger_riesz} can be deduced from more general results for Schr\"odinger operators with nonnegative potentials (see, e.g., \cite{Si}). As before, however, the case $p>2$ is more delicate (see, e.g., \cite{AB,KW,Sh}) and we are not aware of results in the literature for $p>2$ encompassing the exponential potential $V(s) = e^{2s}$ considered here (for example, \cite{AB} requires a reverse H\"older condition on the potential which fails for our exponentially growing $V$, while the assumptions of the recent work \cite{KW} appear to rule out potentials with $\liminf_{|s| \to \infty} V(s) = 0$).

The study of the Riesz transforms associated with the Laplacian $\opL$ on $G = \RR^n \rtimes \RR$ has a discrete counterpart in the analysis of the Riesz transform for a distinguished ``flow Laplacian'' on homogeneous trees. Indeed, in the case $p \leq 2$, $L^p$ bounds for the latter were proved in the same paper \cite{HS} treating the continuous case as well (see also \cite{SaV} for a different proof). The recent work \cite{LMSTV}, tackling the problem of obtaining $L^p$ bounds for $p>2$ in the setting of homogeneous trees, can be thought of as a discrete counterpart to the present paper. As is often the case, while the overall proof strategies in the discrete and continuous settings present several similarities, there are a number of nontrivial issues that are specific to each setting.

\subsection*{Proof strategy}
The first step in our analysis is obtaining relatively explicit formulas for the convolution kernels of the Riesz transforms $\Rz_j$ on $G$, that is, the $X_j$-derivatives of the convolution kernel of the negative fractional power $\opL^{-1/2}$ of the Laplacian. In the case $n=1=j$, a similar approach was adopted in \cite{GS,S}; however, in those works, the formula for $\opL^{-1/2}$ was recovered via representation theory from a formula for the resolvent of the Schr\"odinger operator $\opH = -\partial_s^2+e^{2s}$ on the real line \cite{Hu,T}. Here instead we subordinate $\opL^{-1/2}$ to the heat semigroup $e^{-t\opL}$ and eventually reduce to known heat kernel formulas on real hyperbolic spaces. A crucial part of our analysis, presented in Section \ref{s:sqrtasymp} below, is deriving precise asymptotics (at the origin and at infinity) for the convolution kernel of $\opL^{-1/2}$ and its derivatives.

Once these asymptotics are available, in Section \ref{s:rieszasymp} we analyse the convolution kernels of the adjoint Riesz transforms $\Rz_j^*$, both at the origin and at infinity; in each case we are able to identify a ``main term'', with a simple explicit expression, and a ``remainder term'', which is integrable, thus the boundedness properties of the Riesz transforms are reduced to those of the respective main terms. Unsurprisingly, the behaviour of the main terms at the origin matches that of the corresponding kernels of the Euclidean Riesz transforms for the standard Laplacian on $\RR^{n+1}$, and their boundedness properties are readily established (see Section \ref{s:local}). As may be expected, the challenging part of the problem lies in the study of the kernels at infinity.

Up to this point, our analysis broadly follows the lines (with some variations and additional technical complications) of that in \cite{GS}, where the particular case $n=1=j$ is discussed. However, the main term at infinity of $\Rz_0^*$ does not appear to be amenable to the analysis of \cite{GS}, and here a substantially different approach is developed, which allows us to treat in a uniform way all the adjoint Riesz transforms $\Rz_j^*$ for $j=0,\dots,n$.

Namely, we observe that left-invariant operators on $G = \RR^n_x \rtimes \RR_u$ are invariant under Euclidean translations in the variable $x \in \RR^n$; thus, by writing $L^p(G) = L^p(\RR^n;L^p(\RR))$, we can think of such operators as operator-valued Fourier multiplier operators with respect to the Euclidean Fourier transform on $\RR^n$. The $L^p$-boundedness of the main terms at infinity of the adjoint Riesz transforms is then proved by invoking an operator-valued Fourier multiplier theorem \cite{HvNVW2,SW}. Specifically, the previously obtained formulas allow us to derive quite explicit expressions for the operator-valued symbols of the aforementioned Fourier multiplier operators, and the problem then reduces to verifying suitable R-boundedness properties on $L^p(\RR)$ for such symbols and their derivatives. Eventually, this verification reduces to proving the weighted $L^2$-boundedness for any weight in the Muckenhoupt class $A_2(\RR)$ of certain explicit integral operators on $\RR$. This programme is carried out in Section \ref{s:multiplier} and completes the proof of Theorem \ref{thm:main}.

The above approach via an operator-valued Fourier multiplier theorem, which may be compared to that used in \cite{JST} in relation to Grushin operators, has the drawback of not yielding endpoint boundedness results. In particular, the problem of whether the adjoint Riesz transform $\Rz_0^*$ is of weak type $(1,1)$ remains open. For the adjoint Riesz transforms $\Rz_j^*$ for $j=1,\dots,n$, however, a different approach, based on that in \cite{GS}, can be applied; we discuss this in Section \ref{s:haar} below, thus proving Theorem \ref{thm:wt11}. Apropos endpoint results, what we can establish for \emph{all} the adjoint Riesz transforms $\Rz_j^*$ for $j=0,\dots,n$ is that they are \emph{not} bounded from $H^1(G)$ to $L^1(G)$; this was already proved in \cite{SV} in the case $n=2$, and in Section \ref{s:hardy} we briefly discuss this negative result for any value of $n$.

The fact that left-invariant operators on $G$ are operator-valued Fourier multiplier operators on $\RR^n$ is related to the unitary representation theory of $G$. Indeed, up to equivalence, all the infinite-dimensional irreducible unitary representations of $G$ are obtained by induction from the nontrivial unitary characters of $\RR^n$. Correspondingly, as we illustrate in Section \ref{s:repn}, the operator-valued Fourier multiplier theorem invoked in Section \ref{s:multiplier} can be thought of as a Fourier multiplier theorem for the group Fourier transform on $G$. 
Still in Section \ref{s:repn} we discuss the relation between the Laplacian $\opL$ on $G$ and the Schr\"odinger operator $\opH$, as well as the respective Riesz transforms, and we show how Theorem \ref{thm:main} can be transferred to deduce Theorem \ref{thm:schroedinger_riesz}.

\subsection*{Some open questions}
It would be interesting to investigate whether the results and methods of the present paper could be extended to other settings.

Real hyperbolic spaces are particular cases of symmetric spaces of the noncompact type, as well as of Damek--Ricci spaces, and the distinguished Laplacian $\opL$ considered here has natural analogues in those more general settings (see, e.g., \cite{CGHM,V0}). In those contexts, tools from spherical Fourier analysis are available, which make it possible (see, e.g., \cite{ADY,AO} and references therein) to obtain explicit formulas and asymptotics for many relevant quantities, such as heat kernels and fundamental solutions. For example, in \cite{GS0} $L^p$-boundedness properties of second-order Riesz transforms associated with such distinguished Laplacians are established on arbitrary rank-one symmetric spaces of the noncompact type, and it is a natural question whether similar results could be obtained for first-order Riesz transforms as well.

Another natural setting to consider is that of \cite{MOV,MV}, where the factor $\RR^n$ in the semidirect product $\RR^n \rtimes \RR$ is replaced with an arbitrary stratified Lie group $N$, and correspondingly the Laplacian is replaced with a sub-Laplacian. In \cite{MV} the $L^p$-boundedness for $1 < p \leq 2$ of the natural first-order Riesz transforms on $N \rtimes \RR$ is established, but the problem for $p>2$ is left open. We point out that, differently from the aforementioned case of symmetric spaces, in the sub-Riemannian setting of \cite{MV} no explicit formulas for the heat kernel are available, while such formulas are a fundamental ingredient here. Additionally, here we crucially exploit an operator-valued Fourier multiplier theorem on $\RR^n$, which does not appear to have an obvious counterpart when $\RR^n$ is replaced by a noncommutative stratified group $N$. So a different approach and new ideas would likely be needed in that case.

\subsection*{Notation}

For any two nonnegative quantities $A$ and $B$, we write $A \lesssim B$ to indicate that there is a constant $C \in (0,\infty)$ such that $A \leq C B$. We also write $A \simeq B$ to denote the conjunction of $A \lesssim B$ and $B \lesssim A$. Subscripted variants such as $\lesssim_{a}$ and $\simeq_{a}$ are used to indicate that the implicit constants may depend on a parameter $a$. Moreover, $\RR^*$ stands for $\RR \setminus \{0\}$, and $\RR_+$ for $(0,\infty)$. Finally, $\chi_I$ denotes the characteristic function of a set $I$.

\section{Heat kernel formulas and asymptotics for fractional powers}\label{s:sqrtasymp}

We start by recalling a few useful facts and formulas about the analysis on the group $G = \RR^n \rtimes \RR$. For more details, the reader is referred to \cite{MOV,MV}, where the more general case of $G = N \rtimes \RR$ is discussed, with $N$ a stratified Lie group.

Recall that the group multiplication is given by \eqref{eq:operation}; correspondingly, the group inversion of $G$ is given by
\[
(x,u)^{-1} = (-e^{-u} x,-u),
\]
and the modular function $m$ on $G$ is given by
\begin{equation}\label{eq:modular}
m(x,u) = e^{-nu},
\end{equation}
that is, $e^{-nu} \,dx \,du$ is the left Haar measure on $G$.
Thus, the group convolution is given by
\begin{equation}\label{eq:convolution}
\begin{split}
f * g(x,u) &= \int_G f((x,u) \cdot (x',u')^{-1}) \, g(x',u') \,dx' \,du' \\
&= \int_G f(x-e^{u-u'} x',u-u') \, g(x',u') \,dx' \,du'
\end{split}
\end{equation}
and the $L^1$-isometric involution on $G$ is given by
\begin{equation}\label{eq:involution}
f^*(x,u) = m(x,u) \, \overline{f((x,u)^{-1})} = e^{-nu} \, \overline{f(-e^{-u} x,-u)};
\end{equation}
of course, these formulas make sense for suitably regular functions $f$ and $g$, but can be extended to cover the case where $f$ and $g$ are distributions on $G$.

By the Schwartz Kernel Theorem, any linear left-invariant operator $T$ mapping test functions to distributions on $G$ can be expressed in terms of convolution on the right by a kernel $k_T$, which in general is just a distribution on $G$:
\[
T f = f * k_T.
\]
The adjoint operator $T^*$ is also left-invariant, and $k_{T^*} = k_T^*$. For any $p \in [1,\infty]$, we denote by $Cv^p(G)$ the space of the $L^p$-convolutors of $G$, i.e., the convolution kernels of $L^p(G)$-bounded left-invariant operators.

We equip $G$ with the left-invariant Riemannian metric for which $X_0,\dots,X_n$ is an orthonormal basis. If $d$ is the associated left-invariant distance function on $G$, then
\begin{equation}\label{eq:distance}
\cosh d((x,u),(0,0)) = \cosh u + e^{-u} |x|^2/2.
\end{equation}
We say that a function $f$ on $G$ is radial if the value of $f$ at any point of $G$ only depends on the distance of that point from the origin $(0,0)$. With a slight abuse of notation, if $f$ is a radial function on $G$ and $R \geq 0$, we denote by $f(R)$ the value of $f$ at any point at distance $R$ from the origin.

We point out that $G$, thought of as a Riemannian manifold with the aforementioned structure, is nothing else than the real hyperbolic space of dimension $n+1$, and the Riemannian measure is the left Haar measure $e^{-nu} \,dx \,du$. Moreover, while the Laplacian $\opL$ is not the same as the Laplace--Beltrami $\opL_H$ on the hyperbolic space, the two operators are related. More precisely (see, e.g., \cite[Section IX.1]{VSC}), a power of the modular function $m$ (thought of as a multiplication operator) intertwines the distinguished Laplacian $\opL$ and the shifted Laplace--Beltrami operator:
\begin{equation}\label{eq:intertwine}
\opL f = m^{1/2} (\opL_H - n^2/4) (m^{-1/2} f).
\end{equation}

It is well known that real hyperbolic spaces are rank-one symmetric spaces, and that the group of isometries of a hyperbolic space fixing a point acts transitively on any sphere centred at that point. As the Laplace--Beltrami operator $\opL_H$ is invariant under isometries, the convolution kernels $k_{F(\opL_H)}$ of operators in the functional calculus for $\opL_H$ are radial. This is not the case for the operators $F(\opL)$ in the functional calculus for $\opL$; however, the intertwining relation \eqref{eq:intertwine} between $\opL$ and $\opL_H$ implies that
\begin{equation}\label{eq:intertwine_kernels}
k_{F(\opL)} = m^{1/2} k_{F(\opL_H-n^2/4)},
\end{equation}
thus $m^{-1/2} k_{F(\opL)}$ is radial.

The starting point for our analysis are the following formulas for the heat kernel $h_t = k_{e^{-t\opL}}$ associated with $\opL$ on $G$. 

\begin{prp}\label{prp:heatkernel}
For all $R > 0$, if $n$ is even,
\[
(m^{-1/2} h_t)(R) = \frac{1}{(2\pi)^{n/2}} \left( - \frac{1}{\sinh R} \frac{\partial}{\partial R} \right)^{n/2} h_t^\RR(R),
\]
while, if $n$ is odd,
\begin{multline*}
(m^{-1/2} h_t)(R) \\
= \frac{1}{\sqrt{\pi} (2\pi)^{n/2}} \int_R^\infty \frac{\sinh x}{(\cosh x - \cosh R)^{1/2}} \left( - \frac{1}{\sinh x} \frac{\partial}{\partial x} \right)^{(n+1)/2} h_t^\RR(x) \,dx,
\end{multline*}
where $h_t^{\RR}(x) = (4\pi t)^{-1/2} e^{-x^2/(4t)}$ is the heat kernel on $\RR$.
\end{prp}
\begin{proof}
In light of \eqref{eq:intertwine_kernels}, we know that
\[
m^{-1/2} h_t = e^{(n^2/4)t} k_{e^{-t \opL_H}},
\]
thus the above formulas reduce to those for the heat kernel $k_{e^{-t \opL_H}}$ on real hyperbolic spaces (see, e.g., \cite[eqs.\ (2.2) and (2.3)]{AO} or \cite[eqs.\ (8) and (9)]{CGGM}).
\end{proof}

From the above heat kernel formulas one can derive, via subordination, relatively explicit formulas for the convolution kernel of the fractional power $\opL^{-1/2}$. To this purpose, it is convenient to introduce the notation $Q^0_{\lambda-1/2}$ for the Legendre function of the second kind with parameters $0$ and $\lambda-1/2$, where $\lambda>0$. Note that, according to \cite[Section 3.7, eq.\ (5), p.\ 155]{EMOT1}, for all $z>1$,
\begin{equation}\label{eq:legQ_cpt}
Q^0_{\lambda-1/2}(z) = 2^{-\lambda-1/2} \int_{-1}^1 (z-s)^{-\lambda-1/2} (1-s^2)^{\lambda-1/2} \,ds,
\end{equation}
and also, by \cite[Section 3.7, eq.\ (4), p.\ 155]{EMOT1}, for all $r>0$,
\begin{equation}\label{eq:legQ_cosh}
Q^0_{\lambda-1/2}(\cosh r) = 2^{-1/2} \int_r^\infty e^{-\lambda x} (\cosh x - \cosh r)^{-1/2} \,dx.
\end{equation}

\begin{prp}\label{prp:sqrtexact}
The distribution $k_{\opL^{-1/2}}$ coincides with a smooth function away from the origin, and $m^{-1/2} k_{\opL^{-1/2}}$ is radial. Moreover, for all $R > 0$, if $n$ is even, then
\[
(m^{-1/2} k_{\opL^{-1/2}})(R) = \frac{1}{\pi (2\pi)^{n/2}} \left( - \frac{1}{\sinh R} \frac{\partial}{\partial R} \right)^{(n-2)/2} \frac{1}{R \sinh R},
\]
while, if $n$ is odd, then
\begin{equation}\label{eq:sqrtexact_odd}
\begin{split}
&(m^{-1/2} k_{\opL^{-1/2}})(R) \\
&= \frac{1}{\pi^{3/2} (2\pi)^{n/2}} \left( - \frac{1}{\sinh R} \frac{\partial}{\partial R} \right)^{(n-1)/2} \int_R^\infty (\cosh x-\cosh R)^{-1/2} \frac{dx}{x} \\
&= \frac{\sqrt{2}}{\pi^{3/2} (2\pi)^{n/2}} \left( - \frac{1}{\sinh R} \frac{\partial}{\partial R} \right)^{(n-1)/2} \int_0^\infty Q^0_{\lambda-1/2}(\cosh R) \,d\lambda .
\end{split}
\end{equation}
\end{prp}

In the case $n=1$, an analogous formula can be found in \cite[p.\ 3304, last display]{S}; in the case $n=2$, see \cite[eq.\ (2.1)]{SV}.

\begin{proof}
Recall that, for all $\lambda > 0$,
\[
\lambda^{-1/2} = \frac{1}{\sqrt{\pi}} \int_0^\infty e^{-t\lambda} \frac{dt}{t^{1/2}},
\]
thus, at least formally,
\begin{equation}\label{eq:subordination_sqrt}
k_{\opL^{-1/2}} = \frac{1}{\sqrt{\pi}}  \int_0^\infty h_t \frac{dt}{t^{1/2}}.
\end{equation}
We now observe that, for all $R > 0$,
\[\begin{split}
\int_0^\infty \left(- \frac{1}{\sinh R} \frac{\partial}{\partial R}\right) h_t^\RR(R) \frac{dt}{t^{1/2}} 
&= \frac{1}{4\sqrt{\pi}} \frac{R}{\sinh R} \int_0^\infty e^{-R^2/(4t)} \frac{dt}{t^{2}} \\
&= \frac{1}{\sqrt{\pi}} \frac{1}{R \sinh R} .
\end{split}\]
Thus, for $n$ even, if we plug into \eqref{eq:subordination_sqrt} the formula for the heat kernel $h_t$ from Proposition \ref{prp:heatkernel}, then we obtain
\begin{equation}\label{eq:sub_heat_manipulation}
\begin{split}
(m^{-1/2} k_{\opL^{-1/2}})(R) 
&= \frac{1}{\sqrt{\pi}}  \int_0^\infty (m^{-1/2} h_t)(R) \frac{dt}{t^{1/2}} \\
&= \frac{1}{\sqrt{\pi}} \int_0^\infty \frac{1}{(2\pi)^{n/2}} \left( - \frac{1}{\sinh R} \frac{\partial}{\partial R} \right)^{n/2} h_t^\RR(R) \frac{dt}{t^{1/2}} \\
&= \frac{1}{\pi (2\pi)^{n/2}} \left( - \frac{1}{\sinh R} \frac{\partial}{\partial R} \right)^{(n-2)/2} \frac{1}{R \sinh R},
\end{split}
\end{equation}
as desired.

For $n$ odd, instead, 
we preliminarily observe that, if $g : \RR_+ \to \CC$ is any smooth function decaying sufficiently rapidly at infinity together with its derivatives, then
\[
\int_R^\infty \frac{\sinh x}{(\cosh x-\cosh R)^{1/2}} g(x) \,dx = -2 \int_R^\infty (\cosh x-\cosh R)^{1/2} \frac{\partial}{\partial x} g(x) \,dx
\]
for all $R > 0$, by integration by parts; therefore, differentiating under the integral sign yields the formula
\begin{multline}\label{eq:sinh_byparts}
\left(\frac{1}{\sinh R} \frac{\partial}{\partial R}\right)^N \int_R^\infty \frac{\sinh x}{(\cosh x-\cosh R)^{1/2}} g(x) \,dx \\
= \int_R^\infty \frac{\sinh x}{(\cosh x-\cosh R)^{1/2}} \left(\frac{1}{\sinh x} \frac{\partial}{\partial x}\right)^N g(x) \,dx
\end{multline}
for $N=1$, and by iteration the same holds for any $N \in \NN$.

Now, by arguing as in \eqref{eq:sub_heat_manipulation}, but using the formula for $n$ odd from Proposition \ref{prp:heatkernel}, one obtains
\begin{multline*}
(m^{-1/2} k_{\opL^{-1/2}})(R) \\
= \frac{1}{\pi^{3/2} (2\pi)^{n/2}} \int_R^\infty \frac{\sinh x}{(\cosh x - \cosh R)^{1/2}} \left( - \frac{1}{\sinh x} \frac{\partial}{\partial x} \right)^{(n-1)/2} \frac{1}{x\sinh x} \,dx,
\end{multline*}
and an application of \eqref{eq:sinh_byparts} turns this into the
first expression for $m^{-1/2} k_{\opL^{-1/2}}$ in \eqref{eq:sqrtexact_odd}.
To derive the second one, it is enough to observe that, by \eqref{eq:legQ_cosh},
\[
\sqrt{2} \int_0^\infty Q^0_{\lambda-1/2}(\cosh R) \,d\lambda =  \int_0^\infty (\cosh x - \cosh R)^{-1/2} \frac{dx}{x} ,
\]
which shows that the two expressions are indeed equal.
\end{proof}

From the above formulas we can derive the following asymptotics, which will be crucial for the subsequent analysis.

\begin{prp}\label{prp:sqrtasymp}
For all $R > 0$,
\begin{equation}\label{eq:sqrtasymp}
(m^{-1/2} k_{\opL^{-1/2}})(R) = \frac{1}{\pi (2\pi)^{n/2}} \Phi_n(\cosh R),
\end{equation}
where $\Phi_n : (1,\infty) \to \RR$ is real-analytic and, for all $k \in \NN$, the $k$th derivative of $\Phi_n$ satisfies
\begin{align*}
\Phi_n^{(k)}(X) &= \frac{(-1)^k \Gamma\left(k+\frac{n}{2}\right)}{X^{k+n/2} \log X} \left(1 + O\left(\frac{1}{\log X}\right)\right), \quad\text{as } X \to \infty, \\
\Phi_n^{(k)}(X) &= \frac{(-1)^k \Gamma\left(k+\frac{n}{2}\right)}{2(X-1)^{k+n/2}} \left(1 + O \left((X-1)^{\delta}\right)\right), \quad\text{as } X \to 1^+,
\end{align*}
where $\delta = 1/2$ if $n=1$ and $k=0$, and $\delta = 1$ otherwise.
\end{prp}
\begin{proof}
Let $N = \lfloor (n-1)/2 \rfloor$. Notice that the Chain Rule applied to the change of variables $X = \cosh R$ yields
\begin{equation}\label{eq:change_var_cosh}
\frac{\partial}{\partial X} = \frac{1}{\sinh R} \frac{\partial}{\partial R}.
\end{equation}
Moreover,
\[
\sinh R = \sqrt{X^2 - 1}, \qquad R = \arccosh X = \log( X + \sqrt{X^2-1}).
\]
Thus, from Proposition \ref{prp:sqrtexact} we deduce that \eqref{eq:sqrtasymp} holds with
\[
\Phi_n(X) = (-1)^{N} \Psi^{(N)}_{j}(X),
\]
where $j$ is equal to $0$ or $1$ according to whether $n$ is even or odd, and
\begin{align}
\Psi_0(X) &= \frac{1}{\sqrt{X^2-1} \log(X + \sqrt{X^2-1})}, \label{eq:Psi0}\\
\Psi_1(X) &= \sqrt{\frac{2}{\pi}} \int_0^\infty Q^0_{\lambda-1/2}(X) \,d\lambda \notag\\
&= \pi^{-1/2} \int_{\arccosh X}^\infty (\cosh x - X)^{-1/2} \frac{dx}{x}. \label{eq:Psi1}
\end{align}
Consequently, the proof of the above asymptotics reduces to showing that, for all $k \in \NN$ and $j=0,1$,
\begin{equation}\label{eq:asymp_infty}
\Psi_j^{(k)}(X) = \frac{(-1)^k \Gamma\left(k+1-\frac{j}{2}\right)}{X^{k+1-j/2} \log X} \left(1 + O\left(\frac{1}{\log X}\right)\right), \quad\text{as } X \to \infty,
\end{equation}
and
\begin{equation}\label{eq:asymp_local}
\Psi_j^{(k)}(X) = \frac{(-1)^k \Gamma\left(k+1-\frac{j}{2}\right)}{2 (X-1)^{k+1-j/2}}  \left(1 + O \left((X-1)^{\delta} \right)\right), \quad\text{as } X \to 1^+,
\end{equation}
where $\delta = 1/2$ if $j=1$ and $k=0$, and $\delta = 1$ otherwise.

We first prove \eqref{eq:asymp_infty} in the case $j=0$. Notice that, by \eqref{eq:Psi0}, we can write
\[
\Psi_0(X) = \frac{1}{X \log X} G(1/X,1/\log X),
\]
where
\[
G(a,b) = \frac{(1-a^2)^{-1/2}}{1+b \log(1+\sqrt{1-a^2})}
\]
is a bivariate analytic function in a neighbourhood of $(0,0)$. Thus, from the power series expansion of $G$, we deduce that, for suitable coefficients $c_{\ell,m} \in \RR$,
\[
\Psi_0(X) = \frac{1}{X \log X} \sum_{\ell,m \geq 0} c_{\ell,m} \frac{1}{X^\ell} \frac{1}{\log^m X}
\]
whenever $X$ is sufficiently large. Notice now that
\[
\frac{\partial}{\partial X} \frac{1}{X} = -\frac{1}{X^2}, \qquad \frac{\partial}{\partial X} \frac{1}{\log X} = -\frac{1}{X \log^2 X}.
\]
Thus, by differentiating term by term the above series, we deduce that, for any $k \in \NN$,
\[
\Psi_0^{(k)}(X) = \frac{1}{X^{1+k} \log X} \sum_{\ell,m \geq 0} c^k_{\ell,m} \frac{1}{X^\ell} \frac{1}{\log^m X}
\]
with $c_{0,0}^{k+1} = -(k+1) c^k_{0,0}$. As $c^0_{0,0} = c_{0,0} = G(0,0) = 1$, we conclude that $c_{0,0}^k = (-1)^k k! = (-1)^k \Gamma(k+1)$, thus proving \eqref{eq:asymp_infty} in the case $j=0$.

We now prove \eqref{eq:asymp_local} in the case $j=0$. For this, we observe that
\[
\frac{1}{\sinh R} = \frac{1 + O(R^2)}{R}, \qquad\text{as } R\to 0,
\]
where the term $O(R^2)$ stands for an even analytic function vanishing at $R=0$. Thus, for all $N \in \NN \setminus \{0\}$,
\[\begin{split}
\frac{1}{\sinh R} \frac{\partial}{\partial R} \frac{1+O(R^2)}{R^N}  
&= \frac{1 + O(R^2)}{R} \left[ -N \frac{1+O(R^2)}{R^{N+1}} + \frac{O(R)}{R^N}  \right] \\
&= - N \frac{1+O(R^2)}{R^{N+2}}
\end{split}\]
and therefore, inductively,
\begin{equation}\label{eq:der_sinh}
\left(\frac{1}{\sinh R} \frac{\partial}{\partial R}\right)^k \frac{1}{R \sinh R} = (-1)^k 2^k k! \frac{1+O(R^2)}{R^{2+2k}}, \qquad\text{as } R\to 0.
\end{equation}
As $\Psi_0(\cosh R) = (R \sinh R)^{-1}$ by \eqref{eq:Psi0}, in light of \eqref{eq:change_var_cosh} the previous identity can be rewritten as
\[
\Psi_0^{(k)}(\cosh R) = \frac{(-1)^k 2^k k!}{R^{2+2k}} (1+O(R^2)), \qquad\text{as } R \to 0.
\]
On the other hand, if $X = \cosh R$, we deduce that
\begin{equation}\label{eq:X_R}
X-1 = \frac{R^2}{2}(1+O(R^2))
\end{equation}
and 
\begin{equation}\label{eq:R_X}
R^2 = 2 (X-1)(1+O(X-1)), \qquad\text{as } X \to 1^+,
\end{equation}
whence
\[
\Psi_0^{(k)}(X) = \frac{(-1)^k k!}{2 (X-1)^{1+k}} (1+O(X-1)), \qquad\text{as } X \to 1^+,
\]
which proves \eqref{eq:asymp_local} in the case $j=0$.

Let us now prove \eqref{eq:asymp_infty} in the case $j=1$. According to \eqref{eq:legQ_cpt} and \eqref{eq:Psi1}, we can write
\[
\Psi_1(X) = \pi^{-1/2} \int_0^\infty 2^{-\lambda} \int_{-1}^1 (X-s)^{-\lambda-1/2} (1-s^2)^{\lambda-1/2} \,ds \,d\lambda.
\]
Now, we observe that, for all $k \in \NN$,
\[\begin{split}
\left(\frac{\partial}{\partial X}\right)^k (X-s)^{-\lambda-1/2} &= (-1)^k (X-s)^{-\lambda-1/2-k} \prod_{\ell=0}^{k-1} (\lambda+1/2+\ell) \\
&= (-1)^k (X-s)^{-\lambda-1/2-k} \sum_{\ell=0}^k c^k_\ell \lambda^\ell
\end{split}\]
for suitable rational coefficients $c_\ell^k$, where $c_0^k = 2^{-k} (2k-1)!! = \pi^{-1/2} \Gamma(k+1/2)$. Consequently
\begin{equation}\label{eq:Psi_newexpr}
\begin{split}
&\Psi_1^{(k)}(X)\\
 &= \frac{(-1)^k}{\sqrt{\pi}} \sum_{\ell=0}^k c_\ell^k \int_{-1}^1 (1-s^2)^{-1/2} (X-s)^{-1/2-k} \int_0^\infty \lambda^\ell \left(\frac{2(X-s)}{1-s^2}\right)^{-\lambda} \,d\lambda  \,ds \\
&= \frac{(-1)^k}{\sqrt{\pi}} \sum_{\ell=0}^k \ell! \, c_\ell^k \int_{-1}^1  \frac{(1-s^2)^{-1/2} (X-s)^{-1/2-k}}{\log^{\ell+1} \left(\frac{2(X-s)}{1-s^2}\right)}  \,ds ,
\end{split}
\end{equation}
where we used that $\int_0^\infty \lambda^\ell e^{-\lambda} \,d\lambda = \ell!$ for all $\ell \in \NN$.

Now, under the assumption $|s|<1$, for $X \to \infty$ we have
\[
(X-s)^{-k-1/2} = X^{-k-1/2} (1+O(1/X)),
\]
while
\[
\log \left(\frac{2(X-s)}{1-s^2}\right) = \log X + \log(2/(1-s^2)) + O(1/X)
\]
(notice that $2/(1-s^2) > 2$ and $\log(2/(1-s^2)) > \log 2 > 0$ here),
thus
\[\begin{split}
\frac{1}{\log \left(\frac{2(X-s)}{1-s^2}\right)} &= \frac{1}{\log X} - \frac{\log(2/(1-s^2)) + O(1/X)}{(\log X) \left( \log X + \log(2/(1-s^2)) + O(1/X) \right)} \\
&= \frac{1}{\log X} \left( 1+ \log(2/(1-s^2)) \, O\left(\frac{1}{\log X} \right) \right)
\end{split}\]
and
\[
\frac{(X-s)^{-1/2-k}}{\log^{\ell+1} \left(\frac{2(X-s)}{1-s^2}\right)} 
= \frac{X^{-1/2-k}}{\log^{\ell+1} X} \left( 1+ \log^{\ell+1}(2/(1-s^2)) \, O\left(\frac{1}{\log X} \right) \right).
\]
So
\[\begin{split}
&\int_{-1}^1 \frac{(1-s^2)^{-1/2} (X-s)^{-1/2-k}}{\log^{\ell+1} \left(\frac{2(X-s)}{1-s^2}\right)} \,ds \\
&= \frac{X^{-1/2-k}}{\log^{\ell+1} X} 
\Biggl( \int_{-1}^1 (1-s^2)^{-1/2} \,ds  + O\left(\frac{1}{\log X} \right) \int_{-1}^1 \frac{\log^{\ell+1}(2/(1-s^2))}{(1-s^2)^{1/2}} \,ds  \Biggr) \\
&= \pi \frac{X^{-1/2-k}}{\log^{\ell+1} X} \left( 1 + O\left(\frac{1}{\log X} \right) \right)
\end{split}\]
and therefore, by \eqref{eq:Psi_newexpr},
\[
\Psi_1^{(k)}(X) = (-1)^k \sqrt{\pi} \sum_{\ell=0}^k \ell! \, c_\ell^k \frac{X^{-1/2-k}}{\log^{\ell+1} X} \left( 1 + O\left(\frac{1}{\log X} \right) \right),
\]
which implies the desired asymptotics \eqref{eq:asymp_infty} for $j=1$, as $c_0^k = \pi^{-1/2} \Gamma(k+1/2)$.

We are left with proving \eqref{eq:asymp_local} in the case $j=1$. We start by observing that, by \eqref{eq:Psi1},
\[
\Psi_1(\cosh R) = \pi^{-1/2} \int_R^\infty (\cosh x - \cosh R)^{-1/2} \frac{dx}{x}.
\]
Thus, by \eqref{eq:change_var_cosh} and \eqref{eq:sinh_byparts}, we deduce that
\[\begin{split}
\pi^{1/2} \Psi_1^{(k)}(\cosh R) 
&=  \left(\frac{1}{\sinh R} \frac{\partial}{\partial R} \right)^k \int_R^\infty (\cosh x - \cosh R)^{-1/2} \frac{dx}{x} \\
&= \int_R^\infty \frac{\sinh x}{(\cosh x - \cosh R)^{1/2}} \left(\frac{1}{\sinh x} \frac{\partial}{\partial x}\right)^k \frac{1}{x \sinh x} \,dx.
\end{split}\]

Assume that $R < 1$, and split the above integral as $\int_{R}^\infty = \int_{R}^1 + \int_1^\infty$. Then
\begin{multline*}
\left| \int_1^\infty \frac{\sinh x}{(\cosh x - \cosh R)^{1/2}} \left(\frac{1}{\sinh x} \frac{\partial}{\partial x}\right)^k \frac{1}{x \sinh x} \,dx \right| \\
\leq \int_1^\infty \frac{\sinh x}{(\cosh x - \cosh 1)^{1/2}} \left| \left(\frac{1}{\sinh x} \frac{\partial}{\partial x}\right)^k \frac{1}{x \sinh x} \right| \,dx = O(1).
\end{multline*}
Moreover, for $0 < R < x < 1$,
\[
\cosh x - \cosh R = 2 \sinh((x+R)/2) \sinh((x-R)/2) = \frac{x^2-R^2}{2} (1 + O(x^2))
\]
and, by \eqref{eq:der_sinh},
\[
\left(\frac{1}{\sinh x} \frac{\partial}{\partial x}\right)^k \frac{1}{x \sinh x} = (-1)^k 2^k k! \frac{1+O(x^2)}{x^{2+2k}},
\]
whence
\begin{multline*}
\int_R^1 \frac{\sinh x}{(\cosh x - \cosh R)^{1/2}} \left(\frac{1}{\sinh x} \frac{\partial}{\partial x}\right)^k \frac{1}{x \sinh x} \,dx \\
= (-1)^k 2^{k+1/2} k! \int_R^1 \frac{x^{-1-2k}}{(x^2-R^2)^{1/2}} (1+O(x^2)) \,dx.
\end{multline*}
We notice now that
\[\begin{split}
\int_R^1 \frac{x^{-1-2k}}{(x^2-R^2)^{1/2}} \,dx 
&= R^{-1-2k} \int_1^{1/R} u^{-1-2k} (u^2-1)^{-1/2} \,du \\
&= R^{-1-2k} \left( \int_1^{\infty} - \int_{1/R}^\infty \right) \\
&= C_k R^{-1-2k} (1+ O(R^{1+2k})),
\end{split}\]
where
\[
C_k = \int_1^{\infty} u^{-1-2k} (u^2-1)^{-1/2} \,du =  \frac{1}{2} \int_0^1 s^{k-1/2} (1-s)^{-1/2} \,ds = \frac{\sqrt{\pi}}{2} \frac{\Gamma(k+1/2)}{k!}.
\]
Moreover,
\[\begin{split}
\int_R^1 \frac{x^{-1-2k}}{(x^2-R^2)^{1/2}} x^2 \,dx 
&= R^{1-2k} \int_1^{1/R} u^{1-2k} (u^2-1)^{-1/2} \,du \\
&= R^{-1-2k} O(R^{2-\epsilon}),
\end{split}\]
where $\epsilon = 1$ if $k=0$, and $\epsilon=0$ otherwise. Putting all together finally yields
\[\begin{split}
&\sqrt{\pi} \Psi_1^{(k)}(\cosh R)  \\
&= O(1) +  (-1)^k 2^{k+1/2} k! C_k R^{-1-2k} (1+ O(R^{1+2k}) + O(R^{2-\epsilon} )) \\
&= (-1)^k \sqrt{\pi} 2^{k-1/2} \Gamma(k+1/2) R^{-1-2k} (1+O(R^{2\delta}))
\end{split}\]
for $R < 1$, where $\delta = 1/2$ if $k=0$ and $\delta = 1$ otherwise. In light of \eqref{eq:R_X}, we eventually deduce
\[
\Psi_1^{(k)}(X) = (-1)^k 2^{-1} \Gamma(k+1/2) (X-1)^{-1/2-k} (1+O((X-1)^{\delta}))
\]
as $X \to 1^+$, which proves \eqref{eq:asymp_local} in the case $j=1$.
\end{proof}

\section{Kernel asymptotics for the Riesz transforms}\label{s:rieszasymp}

Simple manipulations of the expression for $k_{\opL^{-1/2}}$ in Proposition \ref{prp:sqrtasymp} allow us to derive the following formulas for the convolution kernels of the Riesz transforms.

\begin{prp}\label{prp:kernel_formulas}
For all $(x,u) \in G \setminus \{(0,0)\}$,
\begin{align*}
k_{\Rz_0-\Rz_0^*}(x,u) &= 2 \sinh u  \frac{m^{1/2}(x,u)}{\pi (2\pi)^{n/2}} \Phi_n'(\cosh R), \\
k_{\Rz_0+\Rz_0^*}(x,u) &= -\frac{m^{1/2}(x,u)}{\pi (2\pi)^{n/2}} \left( n \Phi_n(\cosh R) + e^{-u} |x|^2 \Phi_n'(\cosh R) \right),
\end{align*}
and, for $j=1,\dots,n$,
\begin{align*}
k_{\Rz_j}(x,u) &=  x_j \frac{m^{1/2}(x,u)}{\pi (2\pi)^{n/2}} \Phi_n'(\cosh R),\\
k_{\Rz_j^*}(x,u) &=  - e^{-u} x_j \frac{m^{1/2}(x,u)}{\pi (2\pi)^{n/2}} \Phi_n'(\cosh R), 
\end{align*}
where $R = d((x,u),(0,0))$, and $\Phi_n$ is as in Proposition \ref{prp:sqrtasymp}.
\end{prp}
\begin{proof}
From \eqref{eq:vfs}, \eqref{eq:modular} and \eqref{eq:distance} it is not difficult to derive that, for $j=1,\dots,n$,
\begin{align*}
X_0 m^{1/2}(x,u) &= -\frac{n}{2} m^{1/2}(x,u), &X_0 \cosh R &= \sinh u - e^{-u} |x|^2/2, \\
 X_j m^{1/2}(x,u) &= 0, & \qquad X_j \cosh R &= x_j,
\end{align*}
where $R = d((x,u),(0,0))$.

As a consequence, from \eqref{eq:sqrtasymp} we deduce that
\begin{multline*}
k_{\Rz_0}(x,u) = X_0 k_{\opL^{-1/2}}(x,u) =  -\frac{n}{2} \frac{m^{1/2}(x,u)}{\pi (2\pi)^{n/2}} \Phi_n(\cosh R)\\  + \left( \sinh u - e^{-u} |x|^2/2 \right) \frac{m^{1/2}(x,u)}{\pi (2\pi)^{n/2}} \Phi_n'(\cosh R) 
\end{multline*}
and, for $j=1,\dots,n$,
\[
k_{\Rz_j}(x,u) = X_j k_{\opL^{-1/2}}(x,u) =   x_j \frac{m^{1/2}(x,u)}{\pi (2\pi)^{n/2}} \Phi_n'(\cosh R) .
\]

Now, from \eqref{eq:modular} and \eqref{eq:involution} we immediately derive that
\[
(m^{1/2} f)^*(x,u) = m^{1/2}(x,u) \overline{f(-e^{-u} x, -u)}.
\]
As $d((x,u)^{-1},(0,0)) = d((x,u),(0,0)) = R$, we deduce that
\begin{multline*}
k_{\Rz_0^*}(x,u) = k_{\Rz_0}^*(x,u) =  -\frac{n}{2} \frac{m^{1/2}(x,u)}{\pi (2\pi)^{n/2}} \Phi_n(\cosh R)\\  - \left( \sinh u + e^{-u} |x|^2/2 \right) \frac{m^{1/2}(x,u)}{\pi (2\pi)^{n/2}} \Phi_n'(\cosh R) ,
\end{multline*}
and the required expressions for $k_{\Rz_0-\Rz_0^*}$ and $k_{\Rz_0+\Rz_0^*}$ follow. Moreover,
for $j=1,\dots,n$,
\[
k_{\Rz_j^*}(x,u) = k_{\Rz_j}^*(x,u) =  - e^{-u} x_j \frac{m^{1/2}(x,u)}{\pi (2\pi)^{n/2}} \Phi_n'(\cosh R) ,
\]
as desired.
\end{proof}

Now, by means of the asymptotics in Proposition \ref{prp:sqrtasymp}, we easily deduce precise information on the behaviour of the convolution kernels $k_{\Rz_j}$ in a neighbourhood of the origin, showing that, up to integrable terms, they match the kernels of the Euclidean Riesz transforms for the standard Laplacian on $\RR^{n+1}$.

\begin{prp}\label{prp:local_riesz_asymp}
For $j=0,\dots,n$,
\[
k_{\Rz_j} = -\frac{\Gamma(1+n/2)}{\pi^{1+n/2}} K_j^0 + q_j^0, \qquad k_{\Rz_j^*} = \frac{\Gamma(1+n/2)}{\pi^{1+n/2}} K_j^0 + \tilde q_j^0,
\]
where $q_j^0$ and $\tilde q_j^0$ are locally integrable, while
\[
K_j^0(x,u) = \begin{cases}
(u^2+|x|^2)^{-(n+2)/2} \, u &\text{if } j=0,\\
(u^2+|x|^2)^{-(n+2)/2} \, x_j &\text{otherwise.}
\end{cases}
\]
\end{prp}
\begin{proof}
As the kernels $k_{\Rz_j}$ and $k_{\Rz_j^*}$ are locally integrable off the origin, we only need to analyse their behaviour in a neighbourhood of the origin.

Now, from the asymptotics in Proposition \ref{prp:sqrtasymp} and \eqref{eq:X_R} we deduce that, if $R$ is small, then
\[
\Phi_n(\cosh R) = O(R^{-n}), \qquad \Phi_n'(\cosh R) = - \frac{\Gamma(1+n/2) 2^{n/2}}{R^{n+2}} (1 + O(R^2)).
\]
Moreover, by \eqref{eq:distance}, if $R = d((x,u),(0,0))$, then $|u| \leq R$; so,
by \eqref{eq:modular}, in a neighbourhood of the origin,
\[
e^{-u} = 1 + O(R), \qquad m^{1/2}(x,u) = 1 + O(R), \qquad \sinh u = u (1 + O(R^2)),
\]
and
\[
1 + \frac{R^2}{2}(1+O(R^2)) = \cosh R = \cosh u + e^{-u} |x|^2/2 = 1 + \frac{u^2+|x|^2}{2}(1 + O(R)),
\]
thus
\begin{equation}\label{eq:small_distance}
R^2 = (|x|^2+u^2)(1+O(R)).
\end{equation}
Therefore, from Proposition \ref{prp:kernel_formulas} we deduce that
\begin{align*}
k_{\Rz_j}(x,u) &= - \frac{\Gamma(1+n/2)}{\pi^{1+n/2}} \frac{x_j}{R^{n+2}} + O(R^{-n}), \\
k_{\Rz_j^*}(x,u) &=  \frac{\Gamma(1+n/2)}{\pi^{1+n/2}} \frac{x_j}{R^{n+2}} + O(R^{-n}),
\end{align*}
for $j=1,\dots,n$, and moreover,
\begin{align*}
k_{\Rz_0-\Rz_0^*}(x,u) &= - 2\frac{\Gamma(1+n/2)}{\pi^{1+n/2}} \frac{u}{R^{n+2}} + O(R^{-n}), \\
k_{\Rz_0+\Rz_0^*}(x,u) &=  O(R^{-n}), 
\end{align*}
thus
\begin{align*}
k_{\Rz_0}(x,u) &= - \frac{\Gamma(1+n/2)}{\pi^{1+n/2}} \frac{u}{R^{n+2}} + O(R^{-n}),\\
k_{\Rz_0^*}(x,u) &= \frac{\Gamma(1+n/2)}{\pi^{1+n/2}} \frac{u}{R^{n+2}} + O(R^{-n}).
\end{align*}
As $G$ has dimension $n+1$, the terms $O(R^{-n})$ are locally integrable; moreover, by \eqref{eq:small_distance},
in the above formulas for the $k_{\Rz_j}$ we can replace the denominators $R^{n+2}$ with $(|x|^2+u^2)^{(n+2)/2}$ and obtain the desired expressions.
\end{proof}

It remains to analyse the behaviour of the kernels $k_{\Rz_j}$ at infinity.
By using the asymptotics of Proposition \ref{prp:sqrtasymp}, we will split those kernels into a ``main part'', with a relatively simple expression, and a ``remainder'', which is integrable at infinity. In order to make it easy to recognise those terms that are integrable and therefore can be included in the remainder, it is convenient to record here some integration formulas for radial functions against certain weights on $G$ (cf.\ \cite[Lemma 2.1]{MueT}).

\begin{lem}\label{lem:integrals}
Let $f : \RR \to [0,\infty)$ be measurable. Then
\begin{align}
\int_G \left|x\right| m^{1/2}(x,u) f(R) \,dx \,du &\simeq \int_0^1 f(r) \, r^{n+1} \,dr + \int_1^\infty f(r) \, e^{(1+n/2) r} \,dr , \label{eq:int_x} \\
\int_G \chi_{\{u \leq 1\}} \left|x\right| m^{1/2}(x,u) f(R) \,dx \,du &\simeq \int_0^1 f(r) \, r^{n+1} \,dr + \int_1^\infty f(r) \, e^{(1/2+n/2) r} \,dr , \label{eq:int_x_restu} \\
\int_G \left|\sinh u\right| m^{1/2}(x,u) f(R) \,dx \,du &\simeq \int_0^1 f(r) \, r^{n+1} \,dr + \int_1^\infty f(r) \, e^{(1+n/2) r} \,dr , \label{eq:int_u} \\
\int_G \chi_{\{|u| \leq 1\}} \left|u\right| m^{1/2}(x,u) f(R) \,dx \,du &\simeq \int_0^1 f(r) \, r^{n+1} \,dr + \int_1^\infty f(r) \, e^{nr/2} \,dr,  \label{eq:int_u_restu}
\end{align}
where $R = d((x,u),(0,0))$.
\end{lem}
\begin{proof}
For any $N \in \RR$ and any measurable $w : \RR \to [0,\infty)$,
\[\begin{split}
&\int_G w(u) \left|x\right|^N m^{1/2}(x,u) f(R) \,dx \,du \\
&= \int_\RR \int_{\RR^n} w(u) |x|^N e^{-nu/2} f(\arccosh (\cosh u + e^{-u} |x|^2/2)) \,dx \,du \\
&\simeq \int_\RR \int_0^\infty w(u) \, e^{Nu/2} f(\arccosh (\cosh u + s)) \, s^{(n+N)/2-1} \,ds \,du \\
&= \int_0^\infty f(r) \left[ \sinh r \int_{-r}^r w(u) \, e^{Nu/2} (\cosh r-\cosh u)^{(n+N)/2-1} \,du \right] \,dr ,
\end{split}\]
thus our task is reduced to estimating the term in brackets in the last integral, i.e., the ``radial density'' of the measure $w(u) |x|^N m^{1/2}(x,u) \,dx \,du$, for each of the choices of $w$ and $N$ appearing in the left-hands sides of formulas \eqref{eq:int_x} to \eqref{eq:int_u_restu}.

Recall now that
\[
\cosh r - \cosh u = 2 \sinh \frac{r+|u|}{2} \sinh \frac{r-|u|}{2} \simeq \begin{cases}
e^r &\text{if } |u| \leq r-1,\\
(r-|u|) \sinh r &\text{if } r-1 \leq |u| \leq r.
\end{cases}
\]

Consequently, if $r \leq 1$, then we deduce that
\[\begin{split}
&\sinh r \int_{-r}^r w(u) \, e^{Nu/2} (\cosh r-\cosh u)^{(n+N)/2-1} \,du \\
&\simeq r^{n+N} \int_{-1}^1 w(rv) \, (1-|v|)^{(n+N)/2-1} \,dv,
\end{split}\]
and the latter quantity is indeed comparable to $r^{n+1}$ for each of the four choices of $N$ and $w$ corresponding to the formulas to be proved (indeed, we have $w(u) = 1$ for $|u| \leq 1$ and $N = 1$ in the case of \eqref{eq:int_x} and \eqref{eq:int_x_restu}, or $w(u) \simeq |u|$ for $|u| \leq 1$ and $N = 0$ in the case of \eqref{eq:int_u} and \eqref{eq:int_u_restu}).

If instead $r \geq 1$, then
\[\begin{split}
&\sinh r \int_{-r}^r w(u) \, e^{Nu/2} (\cosh r-\cosh u)^{(n+N)/2-1} \,du \\
&\simeq e^{(n+N)r/2} \int_{|u| \leq r-1} w(u) \, e^{Nu/2} \,du \\
&\qquad+ e^{nr/2} \int_0^1 t^{(n+N)/2-1} w(t-r) \,dt + e^{(n/2+N)r} \int_0^1 t^{(n+N)/2-1} w(r-t) \,dt .
\end{split}\]
Now, if we take $w(u) = 1$ and $N = 1$, then the above quantity is comparable to $e^{(n/2+1)r}$, which completes the proof of \eqref{eq:int_x}. If instead we take $w(u) = \chi_{\{u \leq 1\}}$ and $N = 1$, then the above quantity is comparable to $e^{(n/2+1/2)r}$, thus completing the proof of \eqref{eq:int_x_restu}. Further, if we take $w(u) = \left|\sinh u\right|$ and $N = 0$, then the above quantity is comparable to $e^{(n/2+1)r}$, thus proving \eqref{eq:int_u}. Finally, if we take $w(u) = |u| \chi_{\{|u|\leq 1\}}$ and $N = 0$, then the above quantity is comparable to $e^{nr/2}$, thus completing the proof of \eqref{eq:int_u_restu}.
\end{proof}

For a function $F : \RR^n \to \CC$ and $\lambda>0$, we write $F_{(\lambda)}$ for the rescaled function
\[
F_{(\lambda)}(x) = \lambda^{-n} F(\lambda^{-1} x).
\]
By exploiting the precise asymptotics of Proposition \ref{prp:sqrtasymp}, we can derive the following formulas; in the case $n=1=j$, a similar result can be found in \cite[Lemma 6]{S}.

\begin{prp}\label{prp:kernel_reduction}
We can write
\begin{align*}
k_{\Rz_0-\Rz_0^*} &= - \frac{2 \Gamma(1+n/2)}{\pi^{1+n/2}} (\tilde K_0 + K_0) + q_0, \\
k_{\Rz_j^*} &= - \frac{2 \Gamma(1+n/2)}{\pi^{1+n/2}} K_j + q_j,
\end{align*}
for $j=1,\dots,n$, where $q_0,q_j$ are integrable at infinity, while
\begin{align*}
\tilde K_0(x,u) &= r_0(x) \frac{\chi_{ \{|u| \geq 1\} }}{u} , \\
K_0(x,u) &= \left[ (r_0)_{(e^{u})}(x) - r_0(x) \right]  \frac{\chi_{ \{u \geq 1\} }}{u} , \\
K_j(x,u) &=  r_j(x) \frac{\chi_{ \{u \leq -1\} }}{u}  .
\end{align*}
and
\[
r_0(x) = (1+|x|^2)^{-1-n/2}, \qquad r_j(x) = x_j (1+|x|^2)^{-1-n/2}.
\]
\end{prp}
\begin{proof}
We start with the kernel $k_{\Rz_j^*}$ for $j=1,\dots,n$. From the formula in Proposition \ref{prp:kernel_formulas} and the asymptotics in Proposition \ref{prp:sqrtasymp} we deduce that
\[
k_{\Rz_j}(x,u) = 
 - \frac{\Gamma\left(1+\frac{n}{2}\right)}{\pi (2\pi)^{n/2}} \frac{x_j \, m^{1/2}(x,u)}{(\cosh R)^{1+n/2} \log \cosh R} \left(1 + O\left(\frac{1}{R}\right)\right).
\]
By \eqref{eq:int_x}, the term corresponding to the Big-O in the expression above is integrable at infinity; moreover, by \eqref{eq:int_x_restu}, the part of the above expression where $u \leq 1$ is integrable at infinity too. Thus $k_{\Rz_j}(x,u)$ differs from
\begin{equation}\label{eq:rj_first_asymp}
 -  \frac{\Gamma\left(1+\frac{n}{2}\right)}{\pi (2\pi)^{n/2}} \frac{x_j \, m^{1/2}(x,u)}{(\cosh R)^{1+n/2} \log \cosh R} \chi_{\{u \geq 1\}}
\end{equation}
by a term which is integrable at infinity.

We now observe that, if $u \geq 1$, then, by \eqref{eq:distance},
\[
\cosh R = \frac{1}{2} \left( e^u + e^{-u} + e^{-u} |x|^2 \right) = \frac{e^u}{2} (1 + |e^{-u} x|^2) (1 + O(e^{-2u})),
\]
and moreover
\[
\log \cosh R = u + \log (1 + |e^{-u} x|^2) - \log 2 + O(e^{-2u}),
\]
thus
\[
\begin{split}
\frac{1}{\log \cosh R} &= \frac{1}{u} - \frac{ \log (1 + |e^{-u} x|^2) - \log 2 + O(e^{-2u})}{u \log \cosh R} \\
&= \frac{1}{u} \left( 1 + O\left( \frac{1 + \log(1 + |e^{-u} x|^2)}{u} \right)  \right)
\end{split}
\]
(here we also used that $\cosh R \geq \cosh |u|$, thus $\log\cosh R \gtrsim u$ for $u \geq 1$),
and
\begin{multline}\label{eq:asymp_coshRlogcoshR}
\frac{m^{1/2}(x,u)}{ (\cosh R)^{1+n/2} \log \cosh R}  = 2^{1+n/2} \, \frac{e^{-(1+n)u} (1 + |e^{-u} x|^2)^{-1-n/2} }{u}  \\
\times \left( 1 + O\left( \frac{1 + \log(1 + |e^{-u} x|^2)}{u} \right)  \right),
\end{multline}
in the region where $u \geq 1$.

Consequently, we can rewrite \eqref{eq:rj_first_asymp} as
\begin{multline}\label{eq:rj_second_asymp}
- \frac{2 \Gamma\left(1+\frac{n}{2}\right)}{\pi^{1+n/2}} \frac{e^{-nu} (e^{-u}x_j) (1 + |e^{-u} x|^2)^{-1-n/2} }{u} \chi_{\{u \geq 1\}}  \\
\times \left( 1 + O\left( \frac{1 + \log(1 + |e^{-u} x|^2)}{u} \right)  \right)
\end{multline}
The term corresponding to the Big-O in \eqref{eq:rj_second_asymp} is integrable on $G$, as
\begin{equation}\label{eq:integrability}
\begin{split}
&\int_{1}^\infty \int_{\RR^n} \frac{e^{-nu}}{u} |e^{-u} x|  (1 + |e^{-u} x|^2)^{-1-n/2}   \frac{1 + \log(1 + |e^{-u} x|^2)}{u} \,dx \,du \\
&= \int_{1}^\infty \int_{\RR^n} \frac{1}{u^2} |x|  (1 + |x|^2)^{-1-n/2}  \left(1 + \log(1 + |x|^2)   \right)  \,dx \,du < \infty.
\end{split}
\end{equation}
By taking adjoints (see \eqref{eq:involution}) in the remaining term of \eqref{eq:rj_second_asymp}, we finally obtain that $k_{\Rz_j^*} = k_{\Rz_j}^*$ differs from
\[
- \chi_{\{u \leq -1\}} \frac{2 \Gamma\left(1+\frac{n}{2}\right)}{\pi^{1+n/2}} \frac{1}{u} x_j  (1 + |x|^2)^{-1-n/2}
\]
by a term which is integrable at infinity, as required.

We can analyse $k_{\Rz_0-\Rz_0^*}$ in a similar way. Namely, from the formula in Proposition \ref{prp:kernel_formulas} and the asymptotics in Proposition \ref{prp:sqrtasymp} we deduce that
\[
k_{\Rz_0-\Rz_0^*}(x,u) = 
 - \frac{\Gamma\left(1+\frac{n}{2}\right)}{\pi (2\pi)^{n/2}} 
  \frac{ (2 \sinh u) \, m^{1/2}(x,u)}{(\cosh R)^{1+n/2} \log \cosh R} 
	\left(1 + O\left(\frac{1}{R}\right)\right),
\]
thus, by \eqref{eq:int_u} and \eqref{eq:int_u_restu}, we deduce that $k_{\Rz_0-\Rz_0^*}$ differs from
\begin{equation}\label{eq:r0_first_asymp}
 - \frac{\Gamma\left(1+\frac{n}{2}\right)}{\pi (2\pi)^{n/2}} 
  \frac{ (2 \sinh u) \, m^{1/2}(x,u)}{(\cosh R)^{1+n/2} \log \cosh R} \chi_{\{|u| \geq 1\}} 
\end{equation}
by a term which is integrable at infinity. 

We now restrict our analysis to the region where $u \geq 1$ (indeed, the remaining region $u \leq -1$ can be recovered at the end due to skew-adjointness). If $u \geq 1$, then
\[
2\sinh u = e^u (1 + O(e^{-2u})),
\]
and this, combined with \eqref{eq:asymp_coshRlogcoshR}, allows us to rewrite the part of \eqref{eq:r0_first_asymp} where $u \geq 1$ as
\[
 - \frac{2 \Gamma\left(1+\frac{n}{2}\right)}{\pi^{1+n/2}} 
	 \frac{e^{-nu} (1 + |e^{-u} x|^2)^{-1-n/2} }{u} \chi_{\{u \geq 1\}} 
\left( 1 + O\left( \frac{1 + \log(1 + |e^{-u} x|^2)}{u} \right) \right).
\]
By arguing as in \eqref{eq:integrability}, one sees that the Big-O term in the previous expression gives rise to an integrable term on $G$. Thus we conclude that, in the region where $u \geq 1$, $k_{\Rz_0-\Rz_0^*}$ differs from 
\[
 - \frac{2 \Gamma\left(1+\frac{n}{2}\right)}{\pi^{1+n/2}} 
	 \frac{e^{-nu} (1 + |e^{-u} x|^2)^{-1-n/2} }{u} \chi_{\{u \geq 1\}} \\
\]
by a term which is integrable at infinity.
By taking adjoints (see \eqref{eq:involution}), we also deduce that, in the region where $u \leq -1$, $k_{\Rz_0-\Rz_0^*}$ differs from 
\[
  -\frac{2 \Gamma\left(1+\frac{n}{2}\right)}{\pi^{1+n/2}} 
	 \frac{(1 + | x|^2)^{-1-n/2} }{u} \chi_{\{u \leq -1\}} 
\]
by a term which is integrable at infinity. Thus, by summing the previous two expressions, we deduce that $k_{\Rz_0-\Rz_0^*}$ differs from
\[
-\frac{2 \Gamma\left(1+\frac{n}{2}\right)}{\pi^{1+n/2}} \frac{1}{u} (\chi_{\{u \geq 1\}} e^{-nu} (1 + |e^{-u} x|^2)^{-1-n/2} +\chi_{\{u \leq -1\}} (1 + |x|^2)^{-1-n/2})
\]
by a term which is integrable at infinity. This is easily seen to match the required expression, as
\[\begin{split}
\tilde K_0(x,u) + K_0(x,u) &= r_0(x) \frac{\chi_{ \{|u| \geq 1\} }}{u} + \left[ (r_0)_{(e^{u})}(x) - r_0(x) \right]  \frac{\chi_{ \{u \geq 1\} }}{u} \\
&= \frac{1}{u} \left[ (r_0)_{(e^{u})}(x)  \chi_{ \{u \geq 1\} } +  r_0(x) \chi_{\{u \leq -1\}} \right],
\end{split}\]
and we are done.
\end{proof}

\section{Analysis of the local part and reduction to the part at infinity}\label{s:local}

The formulas in Proposition \ref{prp:local_riesz_asymp} show that the local behaviour of the kernels of the Riesz transforms and their adjoints on $G$ is analogous to that of the corresponding Euclidean kernels on $\RR^{n+1}$, which are standard Calder\'on--Zygmund operators. Based on this, it is not difficult to show, by using the argument described, e.g., in \cite[Lemma 7]{GQS}, that the local part of the kernels $k_{\Rz_j}$ and $k_{\Rz_j^*}$ define operators which are of weak type $(1,1)$ and bounded on $L^p(G)$ for all $p \in (1,\infty)$.
We present here a slightly different approach, based instead on the Calder\'on--Zygmund theory of \cite{HS}. This has for us the technical advantage that it implies at once the boundedness from $H^1(G)$ to $L^1(G)$ too.

We first recall a useful fact about localisation of $L^p$ convolutors, which is true on any Lie group $G$. We include a proof for the reader's convenience.

\begin{lem}\label{lem:cvp_module_cinfty}
Let $p \in [1,\infty]$. Any $\zeta \in C_c^\infty(G)$ is a multiplier of $Cv^p(G)$, i.e., $\zeta K \in Cv^p(G)$ for all $K \in Cv^p(G)$.
\end{lem}
\begin{proof}
In the cases $p = 1$ and $p=\infty$, the space $Cv^p(G)$ is just the space of finite complex measures on $G$, so it is clearly closed under multiplication by elements of $C^\infty_c(G)$.

Assume now $p \in (1,\infty)$. From the results of \cite{C}, it follows that any compactly supported element of the the Fig\`a-Talamanca--Herz algebra $A_p(G)$ is a multiplier of $Cv^p(G)$. (Specifically, in \cite{C} it is shown that $Cv^p(G)$ is the dual of a certain Banach $*$-algebra $\overline{A}_p(G)$ containing all the compactly supported elements of $A_p(G)$, thus $Cv^p(G)$ is naturally a module over $\overline{A}_p(G)$.)
To conclude, it just remains to show that $C^\infty_c(G) \subseteq A_p(G)$. However, this follows immediately from the definition of $A_p(G)$ and the Dixmier--Malliavin Theorem \cite{DM}, allowing us to represent any $\phi \in C^\infty_c(G)$ as a finite sum of functions of the form $\psi * \eta$ with $\psi,\eta \in C_c^\infty(G)$.
\end{proof}

We now revert to the case of the $ax+b$-group $G$ and proceed with the proof of the local boundedness of the Riesz transforms $\Rz_j$ and $\Rz_j^*$.

\begin{prp}\label{prp:local_boundedness}
Let $\zeta \in C_c^\infty(G)$. Then, for $j=0,\dots,n$, the convolution operators with kernels $\zeta k_{\Rz_j}$ and $\zeta  k_{\Rz_j^*}$, as well as their adjoints, are of weak type $(1,1)$, bounded on $L^p(G)$ for all $p \in (1,2]$, and bounded from $H^1(G)$ to $L^1(G)$.
\end{prp}
\begin{proof}
As $(\zeta k_{\Rz_j})^* = (m^{-1} \zeta^*) k_{\Rz_j^*}$ and $(\zeta k_{\Rz_j^*})^* = (m^{-1} \zeta^*) k_{\Rz_j}$, and moreover $m^{-1} \zeta^* \in C^\infty_c(G)$, the boundedness of the adjoints  is reduced to that of the original operators.

We note that, by Proposition \ref{prp:local_riesz_asymp},
\begin{equation}\label{eq:local_reduction}
\zeta  k_{\Rz_j} = -\frac{\Gamma(1+n/2)}{\pi^{1+n/2}} (\zeta K_j^0) + (\zeta q_j^0),
 \quad \zeta k_{\Rz_j^*} = \frac{\Gamma(1+n/2)}{\pi^{1+n/2}} (\zeta K_j^0) + (\zeta \tilde q_j^0),
\end{equation}
and the kernels $\zeta q_j^0$ and $\zeta \tilde q_j^0$ are integrable, so the boundedness properties of $\zeta k_{\Rz_j}$ and $\zeta k_{\Rz_j^*}$ are equivalent to those of $\zeta K_j^0$.  In particular, as $\Rz_j$ and $\Rz_j^*$ are $L^2(G)$-bounded and $Cv^2(G)$ is a $C^\infty_c(G)$-module (see Lemma \ref{lem:cvp_module_cinfty}), we deduce that $\zeta K_j^0 \in Cv^2(G)$ too.

In order to prove that the operator of convolution by $\zeta K_j^0$ has the desired boundedness properties, we shall apply the Calder\'on--Zygmund theory of \cite{HS}. Specifically, we will show that the kernel $\zeta K_j^0$ has a decomposition that satisfies the assumptions of \cite[Theorem 2.3]{MV}.

Consider a smooth nonnegative radial function $\eta$ on $\RR^{n+1}$ supported in an annulus and such that
\[
\sum_{\ell \in \ZZ} \eta(2^\ell ( x,  u)) = 1
\]
for all $(x,u) \in \RR^{n+1} \setminus \{(0,0)\}$. Set $\eta_\ell = \eta(2^{-\ell} \cdot)$ and $\eta_* = \sum_{\ell \leq 0} \eta_\ell$. We can identify $\eta_*$ with a smooth compactly supported function on $\RR^{n+1}$ which is identically $1$ in a neighbourhood of the origin. By an appropriate choice of $\eta$, we may also assume that $\eta_*$ is identically $1$ on the support of $\zeta$, thus $\zeta = \zeta \eta_* = \sum_{\ell \leq 0} \zeta_\ell$, with $\zeta_\ell = \zeta \eta_\ell$. Correspondingly, we decompose
\[
\zeta K_j^0 = \sum_{\ell \leq 0} \zeta_\ell K_j^0 .
\]

Now, recall that, by Proposition \ref{prp:local_riesz_asymp}, the kernels $K_j^0$ are smooth off the origin and homogeneous of degree $-(n+1)$ on $\RR^{n+1}$.
Recall also the vector fields formulas from \eqref{eq:vfs}:
\[
X_0 = \partial_u, \quad X_1 = e^u \partial_{x_1}, \quad \dots, \quad X_n = e^u \partial_{x_n}.
\]
If we define the right-invariant first-order differential operators $X_j^\circ$ by
\[
X_j^\circ f = (X_j f^*)^*,
\]
then it is readily checked (see \eqref{eq:involution}) that
\[
X_0^\circ = - \partial_u - n - x \cdot \nabla_x , \quad X_1^\circ = - \partial_{x_1}, \quad \dots, \quad X_n^\circ = - \partial_{x_n}.
\]
We then deduce that the functions $X_k^\circ K_j^0$ are also smooth off the origin, and moreover
\begin{itemize}
\item $X_k^\circ K_j^0$ is homogeneous of degree $-(n+2)$ on $\RR^{n+1}$ for $k=1,\dots,n$;
\item $X_0^\circ K_j^0$ is the sum of two homogeneous terms of degrees $-(n+1)$ and $-(n+2)$ on $\RR^{n+1}$ respectively.
\end{itemize}
Additionally, by the Leibniz rule,
\[
X_k^\circ (\zeta_\ell K_j^0) = (\tilde X_k^\circ \zeta_\ell) K_j^0 + \zeta_\ell X_k^\circ K_j^0,
\]
where the $\tilde X_k^\circ$ are the nonconstant parts of the $X_k^\circ$, namely,
\[
\tilde X_0^\circ = - \partial_u - x \cdot \nabla_x , \quad \tilde X_1^\circ = - \partial_{x_1}, \quad \dots, \quad \tilde X_n^\circ = - \partial_{x_n},
\]
and clearly, for all $\ell \leq 0$,
\[
\|\zeta_\ell\|_\infty \lesssim 1, \qquad \|\tilde X_k^\circ \zeta_\ell\|_\infty \lesssim 2^{-\ell}.
\]
As a consequence, for all $\ell \leq 0$, by homogeneity considerations,
\[
\| \zeta_\ell K_j^0 \|_\infty \lesssim 2^{-(n+1)\ell}, \qquad \| X_k^\circ(\zeta_\ell K_j^0) \|_\infty \lesssim 2^{-(n+2)\ell}.
\]

Now we observe that, for any $\ell \leq 0$, if $(x,u)$ is in the support of $\zeta_\ell$, then, by \eqref{eq:small_distance}, $R = d((x,u),(0,0)) \simeq |x| + |u| \simeq 2^{\ell}$.
Thus, for any $\ell \leq 0$ and $\epsilon > 0$, 
\[
\int_G | \zeta_\ell K_j^0 (x,u) | (1+ 2^{-\ell} R)^\epsilon \,dx \,du \lesssim_\epsilon 
\int_{|x|+|u|\simeq 2^{\ell}} 2^{-(n+1)\ell} \,dx \,du \simeq 1,
\]
and
\[
\int_G |X_k ((\zeta_\ell K_j^0)^*) (x,u)| \,dx \,du = \int_G |X_k^\circ (\zeta_\ell K_j^0) (x,u)| \,dx \,du \lesssim 2^{-\ell}.
\]

By \cite[Theorem 2.3 and Remark 2.4]{MV}, we conclude that the convolution operator with kernel $\zeta K_j^0$ is of weak type $(1,1)$, bounded on $L^p(G)$ for $p \in (1,2]$, and also bounded from $H^1(G)$ to $L^1(G)$. Due to the reduction \eqref{eq:local_reduction}, the same boundedness properties are shared by the local parts $\zeta k_{\Rz_j}$ and $\zeta k_{\Rz_j^*}$ of the Riesz transform kernels.
\end{proof}

Due to the boundedness of their local parts, the boundedness of the adjoint Riesz transforms is reduced to that of the corresponding parts at infinity, described in Proposition \ref{prp:kernel_reduction}.

\begin{cor}\label{cor:full_reduction}
Let $p \in (1,2]$. For the adjoint Riesz transform $\Rz_j^*$ with $j=1,\dots,n$, any of the following boundedness properties holds if and only if it holds for the convolution operator with kernel $K_j$:
\begin{itemize}
\item $L^p(G)$-boundedness;
\item weak type $(1,1)$;
\item boundedness from $H^1(G)$ to $L^1(G)$.
\end{itemize}
The same is true for the adjoint Riesz transform $\Rz_0^*$ and the kernel $\tilde K_0 + K_0$.
\end{cor}
\begin{proof}
We already know (see \cite{HS,V,MV}) that $\Rz_0$ is of weak type $(1,1)$, bounded from $H^1(G)$ to $L^1(G)$, and bounded on $L^p(G)$ for all $p \in (1,2]$. Thus $\Rz_0^*$ has any of those boundedness properties if and only if $\Rz_0 - \Rz_0^*$ has.

Let us choose $\zeta \in C^\infty_c(G)$ be supported in the region where $|u| < 1$, and be identically one in a neighbourhood of the origin,
and decompose
\[
k_{\Rz_0-\Rz_0^*} = \zeta k_{\Rz_0-\Rz_0^*} + (1-\zeta) k_{\Rz_0-\Rz_0^*}, \qquad k_{\Rz_j^*} = \zeta k_{\Rz_j^*} + (1-\zeta) k_{\Rz_j^*}.
\]
By Proposition \ref{prp:local_boundedness}, we know that the ``local parts'' $\zeta k_{\Rz_0-\Rz_0^*}$ and $\zeta k_{\Rz_j^*}$ have all of the above boundedness properties, thus the problem is reduced to whether the ``parts at infinity'' $(1-\zeta) k_{\Rz_0-\Rz_0^*}$ and $(1-\zeta) k_{\Rz_j^*}$ do.

Now, the support of $\zeta$ is contained in the region $\{ |u| < 1 \}$, and therefore $1-\zeta$ is identically $1$ on $\{|u| \geq 1\}$. So, by Proposition \ref{prp:kernel_reduction}, for $j=1,\dots,n$,
\begin{align*}
(1-\zeta) k_{\Rz_0 - \Rz_0^*} &= - \frac{2 \Gamma(1+n/2)}{\pi^{1+n/2}} (\tilde K_0 + K_0) + (1-\zeta) q_0, \\
(1-\zeta) k_{\Rz_j^*} &= - \frac{2 \Gamma(1+n/2)}{\pi^{1+n/2}} K_j + (1-\zeta) q_j.
\end{align*}
As the kernels $(1-\zeta) q_0, (1-\zeta) q_j$ are integrable on $G$, the boundedness properties of $k_{\Rz_0 - \Rz_0^*}$ and $k_{\Rz_j^*}$ are indeed reduced to those of $\tilde K_0 + K_0$ and $K_j$ respectively.
\end{proof}

As we shall see in Section \ref{s:hardy} below, the boundedness from $H^1(G)$ to $L^1(G)$ fails for the adjoint Riesz transforms, so we set this property apart. As for the remaining boundedness properties, a further reduction is quickly provided by the analysis of the kernel $\tilde K_0$.

\begin{prp}\label{prp:tildeK0}
Let $\tilde K_0$ be as in Proposition \ref{prp:kernel_reduction}. Then the operator
\[
f \mapsto f * \tilde K_0
\]
is of weak type $(1,1)$ and bounded on $L^p(G)$ for all $p \in (1,\infty)$.
\end{prp}
\begin{proof}
By \eqref{eq:convolution},
\[\begin{split}
f * \tilde K_0(x,u) 
&= \int_G f(x-e^{u-u'} x',u-u') \, \tilde K_0(x',u') \,dx' \,du' \\ 
&= \int_\RR \int_{\RR^n} f(x-x',u-u') \, (r_0)_{(e^{u-u'})}(x') \,dx' \, \frac{\chi_{ \{|u'| \geq 1\} }}{u'} \,du' .
\end{split}\]
In other words, if we use the notation
\[
f_x(u) = f^u(x) = f(x,u),
\]
then $f * \tilde K_0 = A B f$, where
\[
(B f)^u = f^u *_{\RR^n} (r_0)_{(e^u)}, \qquad (A f)_x = f_x *_{\RR} k
\]
and $k(u) = \frac{\chi_{ \{|u| \geq 1\} }}{u}$.

Recall that on $G$ with coordinates $(x,u)$ we are using the Lebesgue measure $dx \,du$ as reference measure. As $r_0 \in L^1(\RR^n)$ and the scaling $F \mapsto F_{(\lambda)}$ preserves the $L^1$-norm, it is immediately seen that $B$ is $L^p(G)$-bounded for all $p \in [1,\infty]$.

On the other hand, $k$ is the truncation of a Calder\'on--Zygmund kernel on $\RR$, thus it is well known (see, e.g., \cite[Chapter I, Section 7.1]{St}) that the corresponding convolution operator is of weak type $(1,1)$ and $L^p(\RR)$-bounded for $p \in (1,\infty)$, whence one easily deduces that $A$ is of weak type $(1,1)$ and $L^p(G)$-bounded for $p \in (1,\infty)$.
\end{proof}

\begin{cor}\label{cor:full_reduction_nohardy}
Let $p \in (1,2]$. For the adjoint Riesz transform $\Rz_j^*$ with $j=0,\dots,n$, any of the following boundedness properties holds if and only if it holds for the convolution operator with kernel $K_j$:
\begin{itemize}[beginpenalty=10000,midpenalty=10000]
\item $L^p(G)$-boundedness;
\item weak type $(1,1)$.
\end{itemize}
\end{cor}

\begin{rem}
Formally, the kernels $K_j$ for $j=1,\dots,n$ have a very similar expression to that of $\tilde K_0$ in Proposition \ref{prp:kernel_reduction}. Indeed, by proceeding as in the proof of \ref{prp:tildeK0}, one could write $f * K_j = A_j B_j f$, where $(B_j f)^u = f^u *_{\RR^n} (r_j)_{(e^u)}$ is bounded on $L^p(G)$ for $p \in [1,\infty]$, as $r_j \in L^1(\RR^n)$. However, in this case $(A_j f)_x = f_x *_\RR h$, where the kernel $h(u) = \frac{\chi_{ \{u \leq -1\} }}{u}$ is noncancellative, and indeed the corresponding convolution operator is unbounded on any $L^p$ spaces. Thus, the fact that the composition $A_j B_j$ is bounded on $L^p(G)$ for $p \in (1,\infty)$ is due to a delicate interplay between the two components, and exploits in a fundamental way the cancellative nature of the kernel $r_j$. See also \cite[Remark at the end of Section 3]{GS}.
\end{rem}

\section{The Fourier multiplier approach}\label{s:multiplier}

In this section we complete the proof of Theorem \ref{thm:main}. As discussed, what we need to prove is the $L^p(G)$-boundedness for $p \in (1,2)$ of the adjoint Riesz transforms $\Rz_j^*$ and, by Corollary \ref{cor:full_reduction_nohardy}, this is reduced to the corresponding boundedness properties of the convolution operators with kernels $K_j$.

As it turns out, a convenient way to analyse the $L^p(G)$-boundedness properties of the convolution operators corresponding to the kernels $K_j$, $j=0,\dots,n$, is to write $L^p(G) = L^p(\RR^n; L^p(\RR))$ and to think of the aforementioned operators as operator-valued Fourier multiplier operators on $\RR^n$. Indeed, by definition,  left-invariant operators on $G$ commute, with left translations on $G$; as
\[
(x',0) \cdot (x,u) = (x'+x,u),
\]
this shows in particular that left-invariant operators on $G$ commute with standard translations in the variable $x \in \RR^n$.

More explicitly, for a kernel $K$ on $G$, we can write, at least formally, by \eqref{eq:convolution},
\begin{equation}\label{eq:conv_kernel}
\begin{split}
f * K(x,u) &= \int_G f(x-e^{u-u'} x',u-u') K(x',u') \,dx' \,du' \\
&= \int_G f(x-x',u-u') K_{(e^{u-u'})}(x',u') \,dx' \,du',
\end{split}
\end{equation}
where we are using the notation $K_{(\lambda)}(x,u) = \lambda^{-n} K(\lambda^{-1} x,u)$. So, if $\Four$ denotes the partial Fourier transform in the variable $x \in \RR^n$, then
\[\begin{split}
\Four(f * K)(\xi,u) &= \int_{\RR} (\Four K)(e^{u-u'} \xi,u') (\Four f)(\xi,u-u') \,du' \\
&= [M_K(\xi) (\Four f)(\xi,\cdot)](u),
\end{split}\]
where, for any $\xi \in \RR^n \setminus \{0\}$, $M_K(\xi)$ is the integral operator on $\RR$ given by
\begin{equation}\label{eq:mult_kernel}
[M_K(\xi) \phi ](u) = \int_\RR H_K^{\xi}(u,u') \phi(u') \,du', \quad H_K^{\xi}(u,u') = (\Four K)(e^{u'} \xi,u-u').
\end{equation}

In other words, the partial Fourier transform $\Four$ intertwines the operator $f \mapsto f * K$ with the operator-valued multiplier $M_K$. Thus, in order to prove the $L^p(G)$-boundedness for $p \in (1,\infty)$ of $f \mapsto f * K$, we can make use of a particular instance of an operator-valued Fourier multiplier theorem proved in \cite{SW}. In what follows, $A_2(\RR)$ denotes the Muckenhoupt class of $A_2$ weights on $\RR$, and $[w]_{A_2}$ denotes the $A_2$-characteristic of a weight $w \in A_2(\RR)$.

\begin{prp}\label{prp:opval_fourier_multiplier}
Assume that there exists a nondecreasing function $\psi : [1,\infty) \to [0,\infty)$ such that, for all $w \in A_2(\RR)$, $\alpha \in \{0,1\}^n$, $\xi \in (\RR^*)^n$,
\begin{equation}\label{eq:wbd}
\|\xi^\alpha \partial_\xi^\alpha M_K(\xi)\|_{L^2(w) \to L^2(w)} \leq \psi([w]_{A_2}).
\end{equation}
Then, the operator $f \mapsto f * K$ is bounded on $L^p(G)$ for all $p \in (1,\infty)$.
\end{prp}
\begin{proof}
Let $p \in (1,\infty)$. According to \cite[Proposition 7.5.3]{HvNVW2}, $L^p(\RR)$ is an UMD space with Pisier's contraction property. Therefore, by \cite[Corollary 8.3.22]{HvNVW2}, to deduce the $L^p(G)$-boundedness of $f \mapsto f * K$ it is enough to check that
\[
\{ \xi^\alpha \partial_\xi^\alpha M_K(\xi) \tc \xi \in (\RR^*)^n, \, \alpha \in \{0,1\}^n \}
\]
is R-bounded as a family of operators on $L^p(\RR)$. The required R-boundedness property in turn follows from the assumed uniform $L^2(w)$-bound \eqref{eq:wbd} by \cite[Theorem 8.2.6]{HvNVW2}.
\end{proof}

\begin{rem}
The operator-valued symbol $\xi \mapsto M_K(\xi)$ admits an interpretation in terms of the group Fourier transform of the kernel $K$ on $G$, defined in terms of unitary representation theory; see Remark \ref{rem:fourier_multiplier} below for details. In this sense, the above Proposition \ref{prp:opval_fourier_multiplier} can be thought of as an $L^p$ Fourier multiplier theorem for the group Fourier transform on $G$.
\end{rem}

In order to prove Theorem \ref{thm:main}, we intend to apply Proposition \ref{prp:opval_fourier_multiplier} to the case where $K$ is one of the kernels $K_0,\dots,K_n$. Thus, we need to study the homogeneous $\xi$-derivatives of the Fourier multipliers $M_{K_j}(\xi)$ and the corresponding integral kernels $H_{K_j}^\xi$.

\begin{prp}\label{prp:fmf}
Let $K_0,\dots,K_n$ bet the kernels given in Proposition \ref{prp:kernel_reduction}. Then, for all $\alpha \in \NN^n$, $\xi \in \RR^n \setminus \{0\}$ and $u,u' \in \RR$,
For all $\alpha \in \NN^n$, $j=1,\dots,n$, we have
\begin{align*}
\xi^\alpha \partial_\xi^\alpha H_{K_0}^\xi(u,u') &= \left[ S_{\alpha,0}(e^{u} \xi) - S_{\alpha,0}(e^{u'} \xi) \right] \frac{\chi_{\{u\geq u'+1\}}}{u-u'} , \\
\xi^\alpha \partial_\xi^\alpha H_{K_j}^\xi(u,u') &= S_{\alpha,j}(e^{u'} \xi) \frac{\chi_{\{u\leq u'-1\}}}{u-u'},
\end{align*}
where, for some $\epsilon>0$, the functions $S_{\alpha,0}$ and $S_{\alpha,j}$ satisfy
\begin{align*}
|S_{\alpha,j}(\xi)| &\lesssim_{\alpha} \min\{|\xi|^\epsilon,|\xi|^{-\epsilon}\}, \\
|S_{\alpha,0}(\xi)| &\lesssim_{\alpha} (1+|\xi|)^{-\epsilon}
\end{align*}
for all $\xi \in \RR^n \setminus \{0\}$, as well as
\[
|S_{\alpha,0}(\xi)-S_{\alpha,0}(\xi')| \lesssim_{\alpha} |\xi-\xi'|^{\epsilon} \qquad \text{whenever } |\xi|,|\xi'| \leq 1.
\]
\end{prp}
\begin{proof}
From \eqref{eq:mult_kernel} and Proposition \ref{prp:kernel_reduction}, we deduce that, for $j=0,\dots,n$ and $\alpha \in \NN^n$, the above formulas for the $\xi^\alpha \partial_\xi^\alpha H_{K_j}^\xi$ are satisfied if we set
\[
S_{\alpha,j}(\xi) = \xi^\alpha \partial_\xi^\alpha \hat r_j(\xi) = \hat r_{\alpha,j}(\xi), \qquad r_{\alpha,j}(x) = (-1)^{|\alpha|} \partial_x^\alpha (x^\alpha r_j(x)),
\]
where $\hat f$ denotes the Fourier transform of $f$.
From the expression for $r_0$ in Proposition \ref{prp:kernel_reduction}, we immediately see that $r_0$ is in the H\"ormander symbol class $S^{-(n+2)}(\RR^n)$, therefore so is $r_{\alpha,0}$ for all $\alpha \in \NN^n$. Similarly, for $j=1,\dots,n$, we see that $r_j$ is in the H\"ormander symbol class $S^{-(n+1)}(\RR^n)$ and is an odd function in $x_j$, therefore so is $r_{\alpha,j}$ for all $\alpha \in \NN^n$. The desired properties of the functions $S_{\alpha,j}$, for $j=0,\dots,n$, are then deduced by applying Lemma \ref{lem:symb_ft} below with $r = r_{\alpha,j}$, and taking into account that, for $j>0$, we have $S_{\alpha,j}(0) = 0$ due to the mentioned parity property.
\end{proof}

\begin{lem}\label{lem:symb_ft}
Let $\delta \in (0,1)$. If $r$ is in the symbol class $S^{-(n+\delta)}(\RR^n)$, then its Fourier transform $\hat r$ is $\delta$-H\"older continuous on $\RR^n$ and $\sup_{\xi \in \RR^n} (1+|\xi|)^N |\hat r(\xi)| < \infty$ for all $N \in \NN$.
\end{lem}
\begin{proof}
As $r \in S^{-(n+\delta)}(\RR^n)$ and $\delta >0$, we know that $r$ is integrable on $\RR^n$, thus $\hat r$ is continuous and bounded. Since $(i\xi)^\alpha \hat r(\xi)$ is the Fourier transform of $\partial_x^\alpha r(x)$, which is in $S^{-(n+\delta+|\alpha|)}(\RR^n)$, we deduce that $\xi^\alpha \hat r(\xi)$ is continuous and bounded too for any $\alpha \in \NN^n$, whence the rapid decay of $\hat r$ at infinity.

Now, from the fact that $r \in S^{-(n+\delta)}(\RR^n)$ and from \cite[Proposition 18.2.2 and eq.\ (18.2.7)']{H3} we deduce that $\hat r \in B_{2,\infty}^{\delta+n/2}(\RR^n)$, thus also $\hat r \in B_{\infty,\infty}^{\delta}(\RR^n)$ by \cite[Theorem 6.5.1]{BL}; here $B_{p,q}^s(\RR^n)$ denotes the Besov space on $\RR^n$ of indices $p,q$ and order $s$, as defined in \cite[Definition 6.2.2]{BL}, and we remark that the notation ${}^\infty H_{(s)}$ used in \cite{H3} corresponds to $B_{2,\infty}^s$ here. From the characterisation of $B_{\infty,\infty}^{\delta}(\RR^n)$ given in \cite[Theorem 6.2.5]{BL}, we deduce that $\hat r$ is $\delta$-H\"older continuous, as required.
\end{proof}

In light of the expressions in Proposition \ref{prp:fmf}, the following result implies that the kernels $K_0,K_1,\dots,K_n$ satisfy the assumptions of Proposition \ref{prp:opval_fourier_multiplier}.

\begin{prp}\label{prp:A2bound}
Let $T$ be an integral operator on $\RR$. Assume that either of the following conditions is satisfied.
\begin{enumerate}[label=(\roman*)]
\item\label{en:k_ass1} The integral kernel of $T$ is of the form
\[
(u,u') \mapsto [S(e^u \xi) - S(e^{u'} \xi)] \frac{\chi_{\{u\geq u'+1\}}}{u-u'}
\]
for some $\xi \in \RR^n \setminus \{0\}$ and $S : \RR^n \setminus \{0\} \to \CC$ satisfying
\begin{gather*}
|S(\xi)| \leq (1+|\xi|)^{-\epsilon} \text{ for all } \xi \in \RR^n \setminus \{0\},\\
|S(\xi) - S(\xi')| \leq |\xi-\xi'|^\epsilon \text{ whenever } |\xi|,|\xi'| \leq 1,
\end{gather*}
for some $\epsilon > 0$.

\item\label{en:k_ass2} The integral kernel of $T$ is of the form
\[
(u,u') \mapsto S(e^{u'} \xi) \frac{\chi_{\{u\leq u'-1\}}}{u-u'}
\]
for some $\xi \in \RR^n \setminus \{0\}$ and $S : \RR^n \setminus \{0\} \to \CC$ satisfying
\begin{gather*}
|S(\xi)| \leq \min\{|\xi|,|\xi|^{-1}\}^\epsilon \text{ for all } \xi \in \RR^n \setminus \{0\},
\end{gather*}
for some $\epsilon > 0$.
\end{enumerate}
Then, for any $w \in A_2(\RR)$, the operator $T$ is bounded on $L^2(w)$, with a bound only depending
on $\epsilon$ and $[w]_{A_2}$.
\end{prp}
\begin{proof}
Set either
\[
H_\xi(u,u') = [S(e^u \xi) - S(e^{u'} \xi)] \frac{\chi_{\{u\geq u'+1\}}}{u-u'}
\]
or
\[
H_\xi(u,u') = S(e^{u'} \xi) \frac{\chi_{\{u\leq u'-1\}}}{u-u'},
\]
according to which assumption is satisfied. In either case, we have
\[
H_{\xi}(u,u') = H_{\xi/|\xi|}(u+\log|\xi|,u'+\log|\xi|).
\]
As the class $A_2(\RR)$ and the  $A_2$-characteristic are translation-invariant, we see that it is enough to prove the desired result with $H_{\xi}$ replaced by $H_{\xi/|\xi|}$, i.e., we may assume that $|\xi| = 1$ without loss of generality.

Now, if assumption \ref{en:k_ass2} is satisfied, then
\begin{equation}\label{eq:est_ker_2}
|H_\xi(u,u')| \leq e^{-\epsilon |u'|} \frac{\chi_{\{|u-u'|\geq 1\}}}{|u-u'|}.
\end{equation}
If instead assumption \ref{en:k_ass1} is satisfied, then we obtain different estimates according to the positions of $u,u'$. Namely, if $u \leq 0$, then we also have $u' \leq u-1 \leq -1$ on the support of $H_\xi$; thus both $|e^u \xi|,|e^{u'}\xi| \leq 1$ and
\[
|S(e^u \xi) - S(e^{u'} \xi)| \leq |e^u \xi - e^{u'}\xi|^\epsilon =  (e^u - e^{u'})^{\epsilon} \leq e^{-\epsilon|u|}.
\]
Suppose instead that $u \geq 0$ and $u' \geq 0$; then 
\[
|S(e^u \xi) - S(e^{u'} \xi)| \leq |e^u \xi|^{-\epsilon} + |e^{u'} \xi|^{-\epsilon} = e^{-\epsilon|u|} + e^{-\epsilon|u'|}
\]
Finally, if $u \geq 0$ and $u' \leq 0$, then
\[
|S(e^u \xi) - S(e^{u'} \xi)| \leq 2,
\]
but, at the same time, $|u-u'| = |u|+|u'| \geq \sqrt{u^2 + (u')^2}$. Therefore, by combining the previous estimates, we conclude that
\begin{equation}\label{eq:est_ker_1}
|H_\xi(u,u')| \leq (e^{-\epsilon|u|} + e^{-\epsilon|u'|}) \frac{\chi_{\{|u-u'| \geq 1\}}}{|u-u'|} +  \frac{2}{\sqrt{u^2 + (u')^2}}.
\end{equation}

In light of the estimates \eqref{eq:est_ker_2} and \eqref{eq:est_ker_1}, the desired bound for $T$ follows from Lemma \ref{lem:simple_boundedness} below.
\end{proof}

\begin{lem}\label{lem:simple_boundedness}
Let $\epsilon > 0$. Each of the nonnegative kernels
\begin{align*}
W(u,u') &= \frac{1}{\sqrt{u^2 + (u')^2}}, \\
Z_\epsilon(u,u') &= e^{-\epsilon|u|} \frac{\chi_{\{|u-u'| \geq 1\}}}{|u-u'|}, \\
Z_\epsilon^*(u,u') &= e^{-\epsilon|u'|} \frac{\chi_{\{|u-u'| \geq 1\}}}{|u-u'|}
\end{align*}
defines an integral operator on $\RR$ which is bounded on $L^2(w)$ for all $w \in A_2(\RR)$, with bound depending only on $\epsilon$ and $[w]_{A_2}$.
\end{lem}
\begin{proof}
Let us first discuss the kernel $W$. By Schur's Test, one easily checks that the integral operator with kernel $W$ is $L^2(\RR)$-bounded: indeed, for any $\delta \in (0,1)$,
\[
\int_{\RR} W(u,u') |u'|^{-\delta} \,du' \simeq |u|^{-1} \int_0^{|u|} s^{-\delta} \,ds + \int_{|u|}^\infty s^{-1-\delta} \,ds \simeq_\delta |u|^{-\delta}.
\]
In addition, the kernel $W$ trivially satisfies the pointwise estimates
\begin{align*}
|W(u,u')| &\lesssim \frac{1}{|u-u'|}, \\
|\partial_u W(u,u')| + |\partial_{u'} W(u,u')| &\lesssim \frac{1}{(|u|+|u'|)^2} \leq \frac{1}{|u-u'|^2},
\end{align*}
i.e., $W$ is a standard Calder\'on--Zygmund kernel on $\RR$. The required weighted $L^2$ bounds for $W$ therefore follow from the classical theory of Calder\'on--Zygmund operators (see, e.g., \cite[Chapter V, Section 6.13]{St}).

It remains to discuss the kernels $Z_\epsilon$ and $Z_\epsilon^*$. It is immediately seen that the corresponding integral operators are adjoints of one another; as $A_2(\RR)$ is closed under the mapping $w \mapsto 1/w$, and $[1/w]_{A_2} = [w]_{A_2}$, it is enough to discuss the $L^2(w)$ bound for the kernel $Z_\epsilon$.

By a slight abuse of notation, we write $Z_\epsilon$ to denote the integral operator as well as its integral kernel. Let $w \in A_2(\RR)$. Then, for all $f \in L^2(w)$, by the Cauchy--Schwarz inequality, for all $u \in \RR$,
\[
|Z_\epsilon f(u)| \leq e^{-\epsilon |u|} \|f\|_{L^2(w)} \left( \int_{|u-u'| \geq 1} \frac{du'}{|u-u'|^2 \, w(u')} \right)^{1/2},
\]
and therefore
\[
\| Z_\epsilon f \|_{L^2(w)}^2 \leq \| f \|_{L^2(w)}^2 \int_\RR \int_{|u-u'| \geq 1} \frac{e^{-2\epsilon|u|} w(u)}{|u-u'|^2 \, w(u')} \,du' \,du.
\]
Thus, the desired bound will follow if we can prove that
\[
\int_\RR \int_{|t| \geq 1} \frac{e^{-2\epsilon|u|} w(u)}{|t|^2 \, w(u-t)} \,dt \,du \lesssim_{\epsilon,[w]_{A_2}} 1.
\]

Now, as both $w,1/w \in A_2(\RR)$, by the self-improving property of Muckenhoupt weights (see, e.g., \cite[Corollary 9.2.6]{Gr2}) we deduce that $w,1/w \in A_p(\RR)$ for some $p<2$ only depending on $[w]_{A_2}$. Thus, by the doubling property of Muckenhoupt weights (see, e.g., \cite[Proposition 9.1.5(9)]{Gr2}), we deduce, for all $x \in \RR$ and all $r_2 \geq r_1 > 0$, the estimate
\[
\frac{w((x-r_2,x+r_2))}{w((x-r_1,x+r_1))} \lesssim_{[w]_{A_2}}  \left(\frac{r_2}{r_1}\right)^p,
\]
as well as the analogous estimate with $w$ replaced by $1/w$; here we write $w(I) = \int_I w$ for any Borel set $I \subseteq \RR$, i.e., we identify the weight $w$ with the measure on $\RR$ with density $w$ with respect to the Lebesgue measure. As a consequence, for all $u \in \RR$,
\[\begin{split}
\int_{|t| \geq 1} |t|^{-2} \frac{dt}{w(u-t)} 
&\leq \sum_{k \geq 0} \sum_{\pm} 2^{-2k} \, (1/w)(u\pm [2^k,2^{k+1}]) \\
&\lesssim_{[w]_{A_2}}  \sum_{k \geq 0} 2^{-2k} 2^{kp} (1+|u|)^p \, (1/w)([-1,1]) \\
&\lesssim_{[w]_{A_2}} (1+|u|)^p \, (1/w)([-1,1]),
\end{split}\]
because $p<2$; in the intermediate inequality we used the doubling property for $1/w$. Thus,
\[\begin{split}
\int_\RR \int_{|t| \geq 1} \frac{e^{-2\epsilon|u|} w(u)}{|t|^2 \, w(u-t)} \,dt \,du
&\lesssim_{[w]_{A_2}} (1/w)([-1,1]) \int_{\RR} e^{-2\epsilon|u|} (1+|u|)^p w(u) \,du \\
&\lesssim_{\epsilon,[w]_{A_2}} (1/w)([-1,1]) \, w([-1,1]) \,du \lesssim [w]_{A_2},
\end{split}\]
as desired; here the second inequality follows, via a dyadic decomposition, from the doubling property for $w$, while the last inequality is just a consequence of the definition of the $A_2$ characteristic.
\end{proof}

\begin{proof}[Proof of Theorem \ref{thm:main}]
We already know (see \cite[Theorem 2.4]{HS} and \cite[Theorem 1.1]{MV}) that the Riesz transforms $\Rz_j$ are $L^p$-bounded for $p \in (1,2]$, so it remains to show, by duality, that the adjoint Riesz transforms $\Rz_j^*$ are $L^p$-bounded for $p \in (1,2)$. By Corollary \ref{cor:full_reduction_nohardy}, it is enough to check the analogous boundedness property for the convolution operators with kernels $K_j$. In light of Propositions \ref{prp:fmf} and \ref{prp:A2bound}, the kernels $K_j$ satisfy the assumption of Proposition \ref{prp:opval_fourier_multiplier}, whence the desired $L^p$-boundedness follows.
\end{proof}

\section{The Haar basis approach}\label{s:haar}

Here we aim at proving Theorem \ref{thm:wt11}, i.e., the weak type $(1,1)$ boundedness of the adjoint Riesz transforms $\Rz_j^*$ for $j=1,\dots,n$. The argument presented here is an extension of that in \cite{GS}, which treats the case $n=1$.

We point out that the argument of \cite{GS} was adapted in \cite{LMSTV} to prove the weak type $(1,1)$ boundedness of the adjoints of certain ``horizontal Riesz transforms'' on a homogeneous tree. In the context of the tree, the argument turns out to be particularly clean. However, in the continuous setting of $ax+b$-groups, the proof becomes somewhat more technical due to a number of steps that appear to be needed, as in \cite{GS}, to reduce the problem to a discrete model.

We start with a preliminary reduction, or discretisation.

\begin{lem}\label{lem:discrete_red}
Let $j=1,\dots,n$ and let $r_j$ be as in Proposition \ref{prp:kernel_reduction}. The adjoint Riesz transform $\Rz_j^*$ is of weak type $(1,1)$ if and only if the operator $T_j$ given by
\begin{equation}\label{eq:discrete_op}
T_j f(\cdot,k) = \sum_{h \geq k+1} \frac{\int_0^1 f(\cdot,h+s) *_{\RR^n} (r_j)_{(2^{h+s})} \,ds}{h-k}
\end{equation}
is bounded from $L^1(\RR^n \times \RR)$ to $L^{1,\infty}(\RR^n \times \ZZ)$.
\end{lem}
\begin{proof}
By Corollary \ref{cor:full_reduction}, the weak type $(1,1)$ boundedness of the adjoint Riesz transform is equivalent of that of the convolution operator by the kernel $K_j$.

Much as in the proof of Proposition \ref{prp:tildeK0}, we can write, for $f \in L^1(G)$,
\[
- f * K_j(\cdot,u) = \int_{u+1}^\infty \frac{f(\cdot,u') *_{\RR^n} (r_j)_{(e^{u'})}}{u'-u} \,du';
\]
thus, if $c = \log 2$ and $S f(x,u) = f(x,cu)$, then $S(-f*K_j) = \tilde T_j S f$, where
\[
\tilde T_j f(\cdot,u)
= \int_{u+1/c}^\infty \frac{f(\cdot,u') *_{\RR^n} (r_j)_{(2^{u'})}}{u'-u} \,du',
\]
and therefore the convolution operator by $K_j$ is of weak type $(1,1)$ if and only if the operator $\tilde T_j$ is. Now we write $\tilde T_j = \tilde T_j^1 + \tilde T_j^2$, where
\[
\tilde T_j^1 f(\cdot,u) = \int_{\lfloor u \rfloor+1}^\infty \frac{f(\cdot,u') *_{\RR^n} (r_j)_{(2^{u'})}}{\lfloor u' \rfloor - \lfloor u \rfloor} \,du', 
\]
while $T_j^2$ is an integral operator with kernel $H_j$ given by
\begin{multline*}
H_j((x,u),(x',u')) 
= (r_j)_{(2^{u'})}(x-x') \\
\times \left[ \chi_{\{u + 1/c \leq u'\}} \left(\frac{1}{u'-u}  - \frac{1}{\lfloor u' \rfloor - \lfloor u \rfloor} \right) - \frac{\chi_{\{\lfloor u \rfloor + 1 \leq u' \leq u+1/c\}}}{\lfloor u' \rfloor - \lfloor u \rfloor } \right].
\end{multline*}
As $c<1$, clearly $u' - u \simeq \lfloor u' \rfloor - \lfloor u \rfloor \geq 1$ in the region where $u'\geq u + 1/c$, whence it easily follows that the factor in brackets above is bounded in absolute value by a multiple of $1/(1+(u-u')^2)$;
as moreover $r_j \in L^1(\RR^n)$, this shows that $\sup_{x',u'} \iint |H_j((x,u),(x',u')| \,dx \,du \lesssim 1$, i.e., $\tilde T_j^2$ is bounded on $L^1(G)$. The weak type $(1,1)$ boundedness of $\tilde T_j$ is therefore reduced to that of $\tilde T_j^1$.

In order to conclude, it is enough to observe that
\[
\tilde T_j^1 f(\cdot,u) 
= \sum_{h \geq \lfloor u \rfloor +1} \int_0^1 \frac{f(\cdot,h+s) *_{\RR^n} (r_j)_{(2^{h+s})}}{h - \lfloor u \rfloor} \,ds 
= T_j f(\cdot, \lfloor u \rfloor)
\]
by \eqref{eq:discrete_op}, and therefore $\| \tilde T_j^1 f \|_{L^{1,\infty}(G)} = \| T_j f \|_{L^{1,\infty}(\RR^n \times \ZZ)}$.
\end{proof}

We now recall the key result from \cite{GS} that constitutes the core of the weak type boundedness argument. For this, we need some definitions.

For any half-open interval $I = [a,b) \subseteq \RR$, we write $I^- = [a,(a+b)/2)$ and $I^+ = [(a+b)/2,b)$ for the lower and upper halves of $I$, and write
\[
\psi_I = |I|^{-1} ( \chi_{I^-} - \chi_{I^+} )
\]
for the basic ($L^1$-normalised) Haar function supported on $I$.

Let $\lambda > 0$. A \emph{scale-$\lambda$ partition} is a partition of $\RR$ made of half-open intervals $[a,b)$ of length $\lambda$. A \emph{scale-$\lambda$ Haar-like function} is a function of the form
\[
\sum_{I \in \Part} a_I \psi_I,
\]
where $\Part$ is a scale-$\lambda$ partition and $a_I \in \CC$ for all $I \in \Part$. 
The following result is \cite[Theorem 3]{GS}.

\begin{thm}\label{thm:haar_est}
Let $\beta > 0$. For any $h \in \ZZ$, let $\Delta_h$ be a scale-$\beta 2^{h}$ Haar-like function. Assume that $\sum_{h \in \ZZ} \|\Delta_h\|_{L^1(\RR)} < \infty$. Then
\begin{equation}\label{eq:key_haar_est}
\left| \left\{ (t,k) \in \RR \times \ZZ \tc \left|\sum_{h \geq k+1} \frac{\Delta_h(t)}{h-k} \right| > \alpha \right\} \right| \lesssim \alpha^{-1} \sum_{h \in \ZZ} \|\Delta_h\|_{L^1(\RR)}
\end{equation}
for all $\alpha > 0$, where the implicit constant is absolute.
\end{thm}

While there are similarities between the sum in the left-hand side of \eqref{eq:key_haar_est} and the expression for the operator $T_j$ in \eqref{eq:discrete_op}, a number of further reductions are needed before we can apply Theorem \ref{thm:haar_est} to deduce information on $T_j$.

The following result is essentially a rephrasing of \cite[Lemma 4]{GS}. Here, if $J$ is a countable set, we denote by $\ell \log \ell(J)$ the set of all sequences $a : J \to \CC$ such that
\[
\sum_{j \in J} |a(j)| (1+\left|\log a(j)\right|) < \infty;
\]
moreover, for all $\delta \in (0,\infty)$, we denote by $\ell^\delta(J)$ the set of all sequences $a : J \to \CC$ such that
\[
\sum_{j \in J} |a(j)|^\delta  < \infty.
\]
Clearly
\begin{equation}\label{eq:incl_llogl}
\ell^\delta(J) \subseteq \ell \log \ell(J)
\end{equation}
whenever $\delta \in (0,1)$.

\begin{lem}\label{lem:haar_dec}
There exists a family $\Part = \bigcup_{m \in \ZZ} \Part_m$ of half-open intervals, where $\Part_m$ is a scale-$2^m$ partition for all $m \in \NN$, such that the following hold.
Let $\epsilon > 0$.
Let $\rho \in C^1(\RR)$ satisfy
\[
|\rho(t)| \leq (1+|t|)^{-1-\epsilon}, \qquad |\rho'(t)| \leq (1+|t|)^{-2-\epsilon}
\]
for all $t \in \RR$, and $\int_\RR \rho(t) \,dt = 0$.
Set $c_I(\rho) = |I| \int_\RR \psi_{I}(t) \rho(t) \,dt$ for all $I \in \Part$.
Then
\begin{equation}\label{eq:haar_dec}
\rho = \sum_{I \in \Part} c_{I}(\rho) \psi_{I},
\end{equation}
where the series converges uniformly on $\RR$ and in $L^1(\RR)$. Moreover
\begin{equation}\label{eq:haar_coeff_est}
|c_I(\rho)| \leq \kappa_\epsilon(I)
\end{equation}
for all $I \in \Part$, where the sequence $(\kappa_\epsilon(I))_{I \in \Part}$ depends only on $\epsilon$ and not on $\rho$, and there exists $\delta = \delta(\epsilon) \in (1/2,1)$ such that
\begin{equation}\label{eq:llogl_coeff}
((1+\log_+ |I|)^N \kappa_\epsilon(I))_{I \in \Part} \in \ell^\delta(\Part)
\end{equation}
for all $N \geq 0$.
\end{lem}
\begin{proof}
If we take $\Part_{-m} = \{ D_{mk} \}_{k \in \ZZ}$ for all $m \in \ZZ$, where the $D_{mk}$ are the intervals defined in \cite[Section 1.3]{GS}, then 
\cite[Lemma 4]{GS} gives the decomposition \eqref{eq:haar_dec} and the estimate \eqref{eq:haar_coeff_est} with
\[
\kappa_\epsilon(D_{mk}) = C(\epsilon) 2^{\epsilon m} (1+2^m + |k|)^{-2-\epsilon}.
\]
It only remains to check \eqref{eq:llogl_coeff}. Notice that, for any $N \geq 0$,
\[
(1+\log_+(|D_{mk}|))^N \kappa_\epsilon(D_{mk}) \lesssim_{\epsilon,N} (1+m_-)^N 2^{\epsilon m} (1+2^m + |k|)^{-2-\epsilon},
\]
where $m_- = \max\{-m,0\}$ is the negative part of $m$.
We can now find $\delta = \delta(\epsilon) \in (1/2,1)$ such that $\delta(2+\epsilon) > 2$. As a consequence,
\[\begin{split}
\sum_{m,k \in \ZZ} \left((1+\log_+(|D_{mk}|))^N \kappa_\epsilon(D_{mk})\right)^\delta 
&\lesssim_{\epsilon,N} \sum_{m \in \ZZ} (1+m_{-})^{N\delta} 2^{\epsilon \delta m} (1+2^m)^{1-\delta(2+\epsilon)} \\
&\lesssim_{\epsilon,N} \sum_{m \geq 0} 2^{m(1-2\delta)} + \sum_{m<0} |m|^{N\delta} 2^{-\delta\epsilon|m|} < \infty,
\end{split}\]
thus $((1+\log_+(|D_{mk}|))^N \kappa_\epsilon(D_{mk}))_{m,k \in \ZZ} \in \ell^\delta$.
\end{proof}

We finally recall some addition results for the quasi-Banach space $L^{1,\infty}$ (see, e.g., \cite[Lemma 2.3]{StW} and \cite[Proposition 3]{St0}).

\begin{lem}\label{lem:wtsum}
Let $X$ be a measure space and $N \in \NN \setminus \{0\}$.
\begin{enumerate}[label=(\roman*)]
\item\label{en:wtsum_N} If $F_1,\dots,F_N \in L^{1,\infty}(X)$, then $F_1 + \dots + F_N \in L^{1,\infty}(X)$ and
\[
\left\|\sum_{j=1}^N F_j \right\|_{L^{1,\infty}} \leq 4 (1+\log N) \sum_{j=1}^N \|F_j\|_{L^{1,\infty}}.
\]

\item\label{en:wtsum_llogl} If $F_1,\dots,F_N \in L^{1,\infty}(X)$, and $\|F_j\|_{L^{1,\infty}} \leq A_j$ for some $A_j > 0$ and all $j=1,\dots,N$, with $\sum_{j=1}^N A_j = 1$, then
\[
\left\|\sum_{j=1}^N F_j \right\|_{L^{1,\infty}} \leq 4 + 2 \sum_{j=1}^N A_j \log (1/A_j) .
\]

\end{enumerate}
\end{lem}

We are now ready to prove the weak type bound for the adjoint Riesz transforms.

\begin{proof}[Proof of Theorem \ref{thm:wt11}]
Without loss of generality we assume that $j=1$. We write any $x \in \RR^n$ as $(x_1,x')$, where $x' = (x_2,\dots,x_n) \in \RR^{n-1}$.

By Lemma \ref{lem:discrete_red}, in order to show that $\Rz_1^*$ is of weak type $(1,1)$, it is enough to show the analogous property for the operator $T_1$ defined in \eqref{eq:discrete_op}.

Let $f \in L^1(\RR^n \times \RR)$ be bounded and compactly supported. Set $f^u(x) = f(x,u)$. We now define, for any $u \in \RR$, an approximation of $f^u$ by means of a sequence of measures on $\RR^n$.
Specifically, for any $u \in \RR$, $x'\in \RR^{n-1}$, and $h \in \ZZ$, we define
\begin{equation}\label{eq:approx_measure}
a^u_{h,\ell}(x') = \int_{[\ell 2^h, (\ell+1) 2^h)} f^u(x_1,x') \,dx_1, \qquad \mu_h^{u} = \sum_{\ell \in \ZZ} \delta_{\ell 2^{h}} \otimes a_{h,\ell}^u .
\end{equation}
It is easy to check that
\begin{equation}\label{eq:sum_int}
\sum_{\ell \in \ZZ} |a^u_{h,\ell}(x')| \leq \int_{\RR} |f^u(x_1,x')| \,dx_1
\end{equation}
and that the sequence of measures $\mu_h^{u}$ converges weakly to $f^u$ as $h \to -\infty$; moreover, one can readily prove that
\begin{equation}\label{eq:conv_approx_est}
\|\mu_h^u * \phi - f^u * \phi\|_{L^p(\RR^n)} \leq 2^h \| f^u \|_{L^1(\RR^n)} \| \partial_{x_1} \phi \|_{L^p(\RR^n)}
\end{equation}
for any $p \in [1,\infty]$ and any bounded $\phi \in C(\RR^n)$ with $\partial_{x_1} \phi \in L^p(\RR^n)$.
Thus, if we set
\[
\mu_{h,d}^u = \begin{cases}
\mu_h^u &\text{if } d=0,\\
\mu_{h-d}^u - \mu_{h-d+1}^u&\text{if } d>0,\\
\end{cases}
\]
then we can write, at least in the sense of weak convergence,
\[
f^u = \sum_{d \geq 0} \mu_{\lfloor u \rfloor,d}^{u}
\]
and consequently, for all $h \in \ZZ$,
\begin{equation}\label{eq:decomposition}
\begin{split}
\int_0^1 f^{h+s} *_{\RR^n} (r_1)_{(2^{h+s})} \,ds = \sum_{d \geq 0} \int_0^1 \mu_{h,d}^{h+s} *_{\RR^n} (r_1)_{(2^{h+s})} \,ds.
\end{split}
\end{equation}

We point out that, as $f$ is bounded and compactly supported, the terms $f^{h+s}$ and $\mu_{h,d}^{h+s}$ in \eqref{eq:decomposition} are nonvanishing only for finitely many $h \in \ZZ$, and moreover, for any $h \in \ZZ$, by \eqref{eq:conv_approx_est}, the convergence of the series in the right-hand side of \eqref{eq:decomposition} holds in any $L^p(\RR^n)$ space, $p \in [1,\infty]$. Thus, from \eqref{eq:discrete_op} we deduce that
\begin{equation}\label{eq:T1_approx}
T_1 f(\cdot, k) = \sum_{d \geq 0} \sum_{h \geq k+1} \frac{\int_0^1 \mu_{h,d}^{h+s} *_{\RR^n} (r_1)_{(2^{h+s})} \,ds}{h-k},
\end{equation}
where the sum in $h$ is actually finite, and the sum in $d$ converges in $L^p(\RR^n)$ for any $p \in [1,\infty]$.

Notice now that, from \eqref{eq:approx_measure}, it follows that
\[
a_{h+1,\ell}^u = a_{h,2\ell}^u + a_{h,2\ell+1}^u,
\]
thus
\[\begin{split}
\mu_{h+1}^{u} - \mu_{h}^{u} 
&= \sum_{\ell \in \ZZ} \delta_{\ell 2^{h+1}} \otimes a_{h+1,\ell}^u  - \sum_{\ell \in \ZZ} \delta_{\ell 2^{h}} \otimes a_{h,\ell}^u \\
&= \sum_{\ell \in \ZZ} (\delta_{\ell 2^{h+1}} - \delta_{\ell 2^{h+1} + 2^h}) \otimes a_{h,2\ell+1}^u  \\
&= \tau_{-2^h e_1}\sigma_{h}^u- \sigma_h^u,
\end{split}\]
where
\[
\sigma_h^u \defeq \sum_{\ell \in \ZZ} \delta_{\ell 2^{h+1} + 2^h} \otimes a_{h,2\ell+1}^u,
\]
$e_1 = (1,0,\dots,0) \in \RR^n$, and $\tau_v$ is the translation operator by $v \in \RR^n$.
Therefore, if $d>0$, then
\[\begin{split}
\mu^{h+s}_{h,d} *_{\RR^n} (r_1)_{(2^{h+s})} 
&= (\sigma_{h-d}^{h+s} - \tau_{-2^{h-d} e_1}\sigma_{h-d}^{h+s}) *_{\RR^n} (r_1)_{(2^{h+s})} \\
&= \sigma_{h-d}^{h+s} *_{\RR^n} ((r_1)_{(2^{h+s})} - \tau_{-2^{h-d} e_1} (r_1)_{(2^{h+s})}) \\
&= \sigma_{h-d}^{h+s} *_{\RR^n} ((r_1)_{(2^{s})} -  \tau_{-2^{-d} e_1} (r_1)_{(2^{s})})_{(2^h)}.
\end{split}\]
In conclusion, we have established that, for all $d \geq 0$,
\begin{equation}\label{eq:mu_sigma}
\mu^{h+s}_{h,d} *_{\RR^n} (r_1)_{(2^{h+s})} = \sigma_{h,d}^{h+s} * (R_d^s)_{(2^h)},
\end{equation}
where
\[
\sigma_{h,d}^u = \begin{cases}
\mu_{h}^u &\text{if } d=0,\\
\sigma_{h-d}^u &\text{if } d>0,
\end{cases}
\qquad
R_d^s = \begin{cases}
(r_1)_{(2^s)} &\text{if } d=0,\\
(r_1)_{(2^{s})} -  \tau_{-2^{-d} e_1} (r_1)_{(2^{s})} &\text{if } d>0,
\end{cases}
\]
and actually
\begin{equation}\label{eq:dec_measure}
\sigma_{h,d}^u = \sum_{\ell \in \ZZ} \delta_{\ell 2^{h-d}} \otimes a_{h,d,\ell}^u ,
\end{equation}
where
\begin{equation}\label{eq:new_coeff}
a_{h,d,\ell}^u = \begin{cases}
a_{h-d,\ell}^u &\text{if } d=0 \text{ or } \ell \text{ is odd},\\
0 &\text{otherwise}.
\end{cases}
\end{equation}

From the formula for $r_1$ in Proposition \ref{prp:kernel_reduction} we deduce that $\int_\RR R_0^s(x_1,x') \,dx_1 = 0$ for all $x' \in \RR^{n-1}$ and $s \in [0,1)$, and moreover
\[
|R_0^s(x)| \lesssim (1+|x|)^{-n-1}, \quad |\partial_{x_1} R_0^s(x)| \lesssim (1+|x|)^{-n-2}, \quad |\partial^2_{x_1} R_0^s(x)| \lesssim (1+|x|)^{-n-3}
\]
uniformly in $s \in [0,1)$. Additionally, when $d > 0$, we can write
\[
R_d^s(x) = -2^{-d} \int_{0}^1 \partial_{x_1} R_0^s(x_1+2^{-d}t,x') \,dt.
\]
Consequently the zero-average property $\int_\RR R_d^s(x_1,x') \,dx_1 = 0$ holds for all $d \in \NN$, and there exists $\epsilon > 0$ such that
\begin{align*}
|R_d^s(x)| &\lesssim 2^{-d} (1+|x'|)^{-(n-1)-\epsilon} (1+|x_1|)^{-1-\epsilon} ,\\
|\partial_{x_1} R_d^s(x)| &\lesssim 2^{-d} (1+|x'|)^{-(n-1)-\epsilon} (1+|x_1|)^{-2-\epsilon} 
\end{align*}
uniformly in $d \in \NN$ and $s \in [0,1)$.
Thus, by Lemma \ref{lem:haar_dec}, for all $s \in [0,1)$ and $d \in \NN$, we can write
\begin{equation}\label{eq:dec_ker}
R_d^s = \sum_{I \in \Part} \psi_I \otimes c_{I,d}^s 
\end{equation}
where the functions $c_{I,d}^s = |I| \int_I \psi_I(x_1) R_d^s(x_1,\cdot) \,dx_1$ satisfy, for all $x' \in \RR^{n-1}$,
\begin{equation}\label{eq:est_coeff}
|c_{I,d}^s(x')| \lesssim 2^{-d} (1+|x'|)^{-(n-1)-\epsilon} \kappa_\epsilon(I).
\end{equation}

Consequently, by \eqref{eq:mu_sigma}, \eqref{eq:dec_measure} and \eqref{eq:dec_ker}, we can write
\begin{equation}\label{eq:second_decomposition}
\begin{split}
&\int_{0}^1 \mu^{h+s}_{h,d} *_{\RR^n} (r_1)_{(2^{h+s})} \,ds  \\
&= \sum_{I \in \Part} \sum_{\ell \in \ZZ}  (\delta_{\ell 2^{h-d}} *_\RR (\psi_{I})_{(2^h)})  \otimes \int_0^1 (a_{h,d,\ell}^{h+s} *_{\RR^{n-1}} (c_{I,d}^s)_{(2^h)}) \,ds\\
&= \sum_{I \in \Part} G_{I,d,h} = \sum_{I \in \Part} \sum_{j=0}^{|2^d I|_*-1} G_{I,d,h,j},
\end{split}
\end{equation}
where $|I|_* \defeq \max\{1,|I|\}$, and
\begin{equation}\label{eq:new_haar}
G_{I,d,h,j} = \sum_{\ell \in j+|2^d I|_* \ZZ} \psi_{\ell 2^{h-d} + 2^h I}  \otimes \int_0^1 (a_{h,d,\ell}^{h+s} *_{\RR^{n-1}} (c_{I,d}^s)_{(2^h)}) \,ds.
\end{equation}

We now point out that, for any fixed $x' \in \RR^{n-1}$, the function $G_{I,d,h,j}(\cdot,x')$ is a scale-$2^h|I|$ Haar-like function; indeed, the gap between consecutive indices $\ell$ in the sum in the right-hand side of \eqref{eq:new_haar} is a multiple of $2^d |I|$, thus the intervals $\ell 2^{h-d} + 2^h I$ are all disjoint and contained in a scale-$2^h|I|$ partition.
So from Theorem \ref{thm:haar_est} we deduce that, for any $d \in \NN$, $I \in \Part$, $x' \in \RR^{n-1}$, and $j=0,\dots,|2^d I|_*-1$,
\[\begin{split}
&\left\|\sum_{h \geq k+1} \frac{G_{I,d,h,j}(x_1,x')}{h-k}\right\|_{L^{1,\infty}(\RR_{x_1} \times \ZZ_{k})} \\
&\lesssim \sum_{h \in \ZZ} \|G_{I,d,h,j}(x_1,x')\|_{L^1(\RR_{x_1})}  \\
&\leq  \sum_{h \in \ZZ} \sum_{\ell \in j+|2^d I|_* \ZZ} \int_0^1 |(a_{h,d,\ell}^{h+s} *_{\RR^{n-1}} (c_{I,d}^s)_{(2^h)})(x')| \,ds \\
&\lesssim \sum_{h \in \ZZ} \sum_{\ell \in j+|2^d I|_* \ZZ} 2^{-d} \kappa_\epsilon(I) \left(\int_0^1 |a_{h-d,\ell}^{h+s}| \,ds *_{\RR^{n-1}} ((1+|\cdot|)^{-(n-1)-\epsilon})_{(2^h)}\right)(x'),
\end{split}\]
where \eqref{eq:new_coeff} and \eqref{eq:est_coeff} were used.
Thus, by integration in $x' \in \RR^{n-1}$ and summation in $j$,
\[\begin{split}
&\sum_{j=0}^{|2^d I|_*-1}\left\|\sum_{h \geq k+1} \frac{G_{I,d,h,j}(x)}{h-k}\right\|_{L^{1,\infty}(\RR^n_{x} \times \ZZ_{k})} \\
&\lesssim \sum_{h,\ell \in \ZZ} 2^{-d} \kappa_\epsilon(I) \int_{\RR^{n-1}} \int_0^1 |a_{h-d,\ell}^{h+s}(x')| \,ds \,dx' \\
&= 2^{-d} \kappa_\epsilon(I) \sum_{\ell \in \ZZ} \int_{\RR^{n-1}} \int_\RR |a_{\lfloor u \rfloor-d,\ell}^u(x')| \,du \,dx' \\
&\leq 2^{-d} \kappa_\epsilon(I) \| f \|_{L^1(\RR^n \times \RR)},
\end{split}\]
where Young's inequality was applied to estimate the convolution in $\RR^{n-1}$, and \eqref{eq:sum_int} was used in the last step.
So, by \eqref{eq:second_decomposition} and Lemma \ref{lem:wtsum}\ref{en:wtsum_N}, for all $d \in \NN$ and $I \in \Part$,
\begin{equation}\label{eq:llogl_summable_estimate}
\begin{split}
&\left\|\sum_{h \geq k+1} \frac{G_{I,d,h}(x)}{h-k}\right\|_{L^{1,\infty}(\RR^n_{x} \times \ZZ_{k})} \\
&\lesssim (1+\log_+|2^d I|) 2^{-d} \kappa_\epsilon(I) \| f \|_{L^1(\RR^n \times \RR)} \\
&\lesssim (1+|d|+\log_+|I|) 2^{-d} \kappa_\epsilon(I) \| f \|_{L^1(\RR^n \times \RR)}.
\end{split}
\end{equation}

Now, from \eqref{eq:llogl_coeff}, we deduce that, for some $\delta = \delta(\epsilon) \in (1/2,1)$,
\[\begin{split}
&\sum_{d \in \NN} \sum_{I \in \Part} [(1+|d|+\log_+|I|) 2^{-d} \kappa_\epsilon(I) ]^\delta  \\
&\lesssim_\epsilon \sum_{d \in \NN} (|d| 2^{-d})^\delta \sum_{I \in \Part} (\kappa_\epsilon(I))^\delta + \sum_{d \in \NN} 2^{-\delta d} \sum_{I \in \Part} ((1+\log_+|I|) \kappa_\epsilon(I))^\delta < \infty;
\end{split}\]
in particular, by \eqref{eq:incl_llogl}, 
\[
((1+|d|+\log_+|I|) 2^{-d} \kappa_\epsilon(I) )_{(d,I) \in \NN \times \Part} \in \ell \log \ell(\NN \times \Part).
\]
Thus again, by \eqref{eq:T1_approx}, \eqref{eq:second_decomposition}, \eqref{eq:llogl_summable_estimate} and Lemma \ref{lem:wtsum}\ref{en:wtsum_llogl},
\[\begin{split}
\| T_1 f \|_{L^{1,\infty}(\RR^n \times \ZZ)} 
&= \left\|\sum_{h \geq k+1} \frac{\sum_{I \in \Part, \, d \in \NN} G_{I,d,h}(x)}{h-k}\right\|_{L^{1,\infty}(\RR^n_{x} \times \ZZ_{k})} \\
&\lesssim_\epsilon \| f\|_{L^1(\RR^n \times \RR)}.
\end{split}\]
As bounded and compactly supported functions $f$ are dense in $L^1(\RR^n \times \RR)$, this shows that $T_1$ is of weak type $(1,1)$, as desired.
\end{proof}

\section{Hardy space unboundedness}\label{s:hardy}

Here we prove the following negative result, extending the analogous one in \cite[Theorems 4.2 and 5.2]{SV} for the case $n=2$. Our proof below appears to be somewhat shorter than the proofs in \cite{SV}, but of course here we can rely on the various reductions and asymptotics established above.

\begin{prp}
For $j=0,\dots,n$, the adjoint Riesz transform $\Rz_j^*$ is not bounded from $H^1(G)$ to $L^1(G)$.
\end{prp}
\begin{proof}
In light of Corollary \ref{cor:full_reduction}, for $j=1,\dots,n$, in order to show that $\Rz_j^*$ is not bounded from $H^1(G)$ to $L^1(G)$, it is enough to prove the same for the convolution operator with kernel $K_j$. To this purpose, we will exhibit (multiples of) atoms $a_j$ of $H^1(G)$ such that $a_j * K_j \notin L^1(G)$.

Let $\phi \in C^\infty_c(\RR^n)$ and $\psi \in C^\infty_c(\RR)$ be nontrival nonnegative cutoffs, and let us consider, for any $v \in \RR^n$,
\[
a_v(x,u) = (\phi(x+v) - \phi(x)) \psi(u).
\]
Clearly $a_v$ is bounded and compactly supported, and $\int_G a_v = 0$, thus $a_v$ is a multiple of an $H^1(G)$ atom (see, e.g., \cite[Definitions 3.4 and 3.10]{MOV}). Moreover, by \eqref{eq:conv_kernel} and Proposition \ref{prp:kernel_reduction},
\begin{equation}\label{eq:atom_conv}
\begin{split}
&a_v * K_j(x,u) \\
&= \int_G a_v(x-x',u-u') (K_j)_{(e^{u-u'})}(x',u') \,dx' \,du' \\
&= \int_G \phi(x-x') \psi(u-u') [(r_j)_{(e^{u-u'})}(x'+v) - (r_j)_{(e^{u-u'})}(x')]  \frac{\chi_{\{u' \leq -1\}}}{u'} \,dx' \,du'.
\end{split}
\end{equation}

Now, by the Fundamental Theorem of Calculus,
\begin{equation}\label{eq:lagrange}
(r_j)_{(\lambda)}(y+v)-(r_j)_{(\lambda)}(y) = \lambda^{-1} \int_0^1 (v \cdot \nabla r_j)_{(\lambda)} (y + t v) \,dt,
\end{equation}
and moreover
\[
\nabla r_j(x) = (1+|x|^2)^{-1-n/2} e_j - (2+n) x_j (1+|x|^2)^{-2-n/2} x,
\]
thus
\[
e_j \cdot \nabla r_j(x) = (1+|x|^2)^{-1-n/2}  - (2+n) x_j^2 (1+|x|^2)^{-2-n/2}.
\]
In particular, if $x$ is in the region $V_j \subseteq \RR^n$ where $x_j^2 \geq (1+|x|^2)/2$, then
\begin{equation}\label{eq:der_rj}
e_j \cdot \nabla r_j(x) \leq -(n/2) (1+|x|^2)^{-1-n/2}.
\end{equation}
Clearly, we can find a nonnegligible region $U_j \subseteq \RR^n$ such that, if $x \in U_j$, $|x-x'| \leq 2$ and $|u| \leq 1$, then $e^u x' \in V_j$;
for example, one can take $U_j = \{ x \in \RR^n \tc |x_j| \geq 100, \, |x|^2 - x_j^2 \leq 1 \}$.
 As a consequence, if $x \in U_j$, $|x-x'| \leq 1$ and $|u| \leq 1$, then, by \eqref{eq:lagrange} and \eqref{eq:der_rj},
\[
-[(r_j)_{(e^u)}(x'+e_j) - (r_j)_{(e^u)}(x')] \gtrsim (1+|x|^2)^{-1-n/2}.
\]
Thus, if we choose $\phi$ and $\psi$ supported in balls of radius $1$ centred at the origin, and we take $x \in U_j$ and $u \leq -2$, then, by \eqref{eq:atom_conv},
\[
a_{e_j} * K_j(x,u) \gtrsim (1+|x|^2)^{-1-n/2} |u|^{-1},
\]
and clearly the right-hand side is not integrable on $U_j \times (-\infty,-2]$. We can then take $a_j = a_{e_j}$ to get the desired result.

The argument for $\Rz_0^*$ is analogous. Here, by Corollary \ref{cor:full_reduction}, it is enough to find a multiple $a_0$ of an $H^1(G)$-atom such that $a_0 * (\tilde K_0 + K_0) \notin L^1(G)$. If we take $a_0 = a_v$ as before, with $\phi,\psi$ supported in centred balls of radius $1$, then from \eqref{eq:conv_kernel} we deduce that $a_v * K_0(x,u) = 0$ whenever $u \leq -2$, as $K_0$ is supported in the region where $u \geq 1$. Thus, by restricting to $u \leq -2$, we have
\[\begin{split}
&a_v * (\tilde K_0 + K_0)(x,u) \\
&=a_v * \tilde K_0(x,u) \\
&= \int_G \phi(x-x') \psi(u-u') [(r_0)_{(e^{u-u'})}(x'+v) - (r_0)_{(e^{u-u'})}(x')]  \frac{\chi_{\{u' \leq -1\}}}{u'} \,dx' \,du'.
\end{split}\]
So we can proceed with the same argument as before with $r_j$ replaced by $r_0$. As
\[
\nabla r_0(x) = -(2+n) (1+|x|^2)^{-2-n/2} x,
\]
here one can just take as $v$ any unit vector, and as $V_0 \subseteq \RR^n$ the region where $x \cdot v \geq \sqrt{(1+|x|^2)/2}$, in order to show that
\[
a_v * (\tilde K_0 + K_0)(x,u) \gtrsim (1+|x|^2)^{-3/2-n/2} |u|^{-1}
\]
on a region of the form $U_0 \times (-\infty,-2]$, and conclude that $a_v * (\tilde K_0 + K_0)$ is not integrable.
\end{proof}

\section{Representation theory of \texorpdfstring{$G$}{G} and Riesz transforms for a one-dimensional Schr\"odinger operator}\label{s:repn}

The following proposition collects a few basic results about the representation theory of the group $G$ (cf., e.g., \cite{Kh,Hu,ET} and \cite[Section 6.7]{F} for the case $n=1$).
Here, for a unitary representation $\pi$ of $G$ on a Hilbert space $H$ and a function $F \in L^1(G)$, we write $\pi(F) = \int_G F(x,u) \,\pi((x,u)^{-1}) \,dx \,du$ for the bounded operator on $H$ corresponding to $F$ via $\pi$.

\begin{prp}\label{prp:irrep}
Let $\xi \in \RR^n \setminus \{0\}$.
\begin{enumerate}[label=(\roman*)]
\item\label{en:irrep_formula} The formula
\[
\sigma^\xi(x,u) \phi(s) = e^{i e^{s} \xi \cdot x} \phi(s+u)
\]
defines a unitary representation of $G$ on $L^2(\RR)$, which also acts by isometries on $L^p(\RR)$ for all $p \in [1,\infty]$.
\item\label{en:irrep_transl} If $\tau_v \phi(s) = \phi(s-v)$ denotes the translation operator by $v$ on $L^2(\RR)$, then
\[
\sigma^{e^v \xi} = \tau_{-v} \sigma^{\xi} \tau_{v}.
\]
\item\label{en:irrep_integralop} For all $K \in L^1(G)$,
\[
\sigma^\xi(K) = M_K(\xi),
\]
where $M_K(\xi)$ is the integral operator with kernel $H_K^\xi$ defined in \eqref{eq:mult_kernel}.
\item\label{en:irrep_transf} For all $K \in L^1(G)$ and $p \in [1,\infty]$,
\[
\|M_K(\xi)\|_{L^p(\RR) \to L^p(\RR)} \leq \|K\|_{Cv^p(G)},
\]
and in particular
\[
\|M_K(\xi)\|_{L^p(\RR) \to L^p(\RR)} \leq \|K\|_{L^1(G)}.
\]
\end{enumerate}
\end{prp}
\begin{proof}
It is readily checked that $\sigma^\xi$ is a representation of $G$, and it is evident from the given expression that it acts isometrically on $L^p(\RR)$ for any $p \in [1,\infty]$; this proves part \ref{en:irrep_formula}. It is similarly easily checked that translations intertwine representations $\sigma^\xi$ with parameters that are positive multiples of each other, as described in part \ref{en:irrep_transl}.

Now, if $K \in L^1(G)$, then, for all $\phi \in L^2(\RR)$,
\[\begin{split}
\sigma^\xi(K) \phi(s)
&= \int_G K(x,u) \sigma^\xi((x,u)^{-1}) f(s) \,dx \,du \\
&= \int_\RR \int_{\RR^n} K(x,u) e^{-ie^{s-u} \xi \cdot x} \phi(s-u) \,dx \,du \\
&= \int_\RR \int_{\RR^n} K(x,s-s') e^{-ie^{s'} \xi \cdot x} \phi(s') \,dx \,ds' \\
&= \int_\RR \Four K(e^{s'} \xi,s-s') \phi(s') \,ds',
\end{split}\]
where $\Four$ is the partial Fourier transform in the variable $x$. Comparing the above formula with \eqref{eq:mult_kernel} proves part \ref{en:irrep_integralop}. 

In light of the identity in part \ref{en:irrep_integralop}, the second estimate in part \ref{en:irrep_transf} is an immediate consequence of the fact that $\sigma^\xi$ acts isometrically on $L^p(\RR)$. The first, more precise estimate follows from the transference principle (see, e.g., \cite[Theorem 2.4]{CW} or \cite[Theorem 2.7]{BPW}), as $G$ is an amenable group since it is solvable (see \cite[Corollary 13.5]{P}). 
\end{proof}

\begin{rem}\label{rem:mackey}
The unitary representations $\sigma^\xi$ (where $\xi \in \RR^n \setminus \{0\}$) on $L^2(\RR)$ introduced in Proposition \ref{prp:irrep} are nothing else than the representations of $G$ induced by the nontrivial characters $x \mapsto e^{i \xi \cdot x}$ of the normal subgroup $\RR^n$. Another family of unitary representations of $G$ is obtained by lifting to $G$ the characters of $\RR = G / \RR^n$, i.e., is given by the characters $\chi_\mu(x,u) = e^{i \mu u}$ (where $\mu \in \RR$) of $G$. By the ``Mackey machine'' (see, e.g., \cite[Theorem 6.43]{F}) all these representations are irreducible and any irreducible unitary representation of $G$ is equivalent to one of these; moreover, the only nontrivial equivalences between these representations are those described in Proposition \ref{prp:irrep}\ref{en:irrep_transl}. This corresponds to the fact that the coadjoint action of $G$ on (the Pontryagin dual of) $\RR^n$ factors through the normal abelian subgroup $\RR^n$ and corresponds to the action by dilations $\xi \mapsto e^u \xi$ of $\RR_u$ on $\RR^n_\xi$; the orbits are therefore $\{0\}$ and the half-lines $\RR_+ \xi$ for $\xi \in \RR^n \setminus \{0\}$.
\end{rem}

\begin{rem}\label{rem:fourier_multiplier}
By Remark \ref{rem:mackey}, the unitary representations $\sigma^\xi$ for $\xi \in \RR^n \setminus \{0\}$ exhaust (up to equivalence) all the infinite-dimensional irreducible representations of $G$. We can therefore think of the mapping $K \mapsto (\sigma^\xi(K))_{\xi \in \RR^n \setminus \{0\}}$ as a ``concrete'' realisation of the group Fourier transform on $G$ (where we neglect one-dimensional representations and allow for some redundancy in the parametrisation of the infinite-dimensional ones).
The identity $\sigma^\xi(K) = M_K(\xi)$ discussed in Proposition \ref{prp:irrep}\ref{en:irrep_integralop} above thus yields a potentially suggestive interpretation of Proposition \ref{prp:opval_fourier_multiplier}: namely, Proposition \ref{prp:opval_fourier_multiplier} can be thought of as an $L^p$ Fourier multiplier theorem for the group Fourier transform of $G$, where $L^p(G)$-boundedness properties of $f \mapsto f * K$ are deduced from suitable boundedness properties of the ``symbol'' $\xi \mapsto \sigma^\xi(K)$ and its derivatives.
In these respects, Proposition \ref{prp:opval_fourier_multiplier} can be compared to results of a similar flavour appeared in the literature for other Lie groups (see, e.g., \cite{CW0,DeMa,FiRu} and references therein).
\end{rem}

In what follows, we write $\Sz(G)$ for the usual Schwartz class on $\RR^n \times \RR$ and $\Sz'(G)$ for the corresponding dual space of tempered distributions.

\begin{lem}\label{lem:partiallytempered}
$Cv^2(G) \subseteq \Sz'(G)$. Moreover, if a sequence $K_n$ in $Cv^2(G)$ converges to $K \in Cv^2(G)$ with respect to the weak operator topology on $L^2(G)$, then $K_n \to K$ in $\Sz'(G)$ too.
\end{lem}
\begin{proof}
Of course $\Sz(G) \subseteq L^2(G)$, as we are using the standard Lebesgue measure $dx \,du$ on $G$. Additionally, from \eqref{eq:operation} it is not difficult to see that $\Sz(G)$ is invariant under left translations (indeed, left $G$-translations acts as affine linear maps on $\RR^n \times \RR$). Moreover, the right-invariant vector fields corresponding to \eqref{eq:vfs} are
\[
X_0^r = \partial_u + x \cdot \nabla_x, \quad X_1^r = \partial_{x_1}, \quad \dots \quad, X_n^r = \partial_{x_n},
\]
thus $\Sz(G)$ is clearly invariant under right-invariant differential operators.

Let $K \in Cv^2(G)$. 
Now, for any $f \in \Sz(G)$ and right-invariant differential operator $D$ on $G$, we have $Df \in \Sz(G) \subseteq L^2(G)$, thus also $D(f*K^*) = (Df)* K^* \in L^2(G)$, because $K^* \in Cv^2(G)$. By Sobolev's embedding we conclude that $f*K^*$ is smooth, and moreover, if $e = (0,0)$ is the identity element, then we can bound
\[
|\langle f , K \rangle| = |\langle f * K^* , \delta_e \rangle| = |(f*K^*)(e)| \lesssim \sum_{D} \| (Df) * K^* \|_2 \lesssim \sum_{D} \| Df \|_2 ,
\]
where the sum is extended over a suitable finite family (independent of $f$) of right-invariant differential operators $D$. The above estimate shows that $K$ indeed defines a bounded linear functional on $\Sz(G)$, i.e., $K \in \Sz'(G)$.

By the discussion at the beginning of the proof, we have a continuous representation of $G$ on the Frech\'et space $\Sz(G)$, where $G$ acts by left translations, and moreover all elements of $\Sz(G)$ are smooth vectors for this representation, as $\Sz(G)$ is invariant by right-invariant differential operators. If we apply \cite[Th\'eor\`eme 3.3]{DM} to this representation, then we conclude that any $\phi \in \Sz(G)$ can be written as a finite sum of the form
\begin{equation}\label{eq:DMdec}
\phi = \sum_j \psi_j * \phi_j,
\end{equation}
where $\psi_j \in C^\infty_c(G)$ and $\phi_j \in \Sz(G)$.

Let now $K_n \in Cv^2(G)$ be a sequence converging to $K \in Cv^2(G)$ in the sense of the weak operator topology on $L^2(G)$. Let $\phi \in \Sz(G)$, and decompose $\phi$ as in \eqref{eq:DMdec}. Then
\begin{equation}\label{eq:duality_dec}
\langle K_n , \phi \rangle = \sum_j \langle K_n , \psi_j * \phi_j \rangle = \sum_j \langle \overline{\check\psi_j} * K_n , \phi_j \rangle,
\end{equation}
where $\check\psi_j(x,u) = \psi_j((x,u)^{-1}) = \psi_j(-e^{-u} x,-u)$. So $\overline{\check\psi_j} \in C^\infty_c(G)$ and therefore both $\overline{\check\psi_j},\phi_j \in L^2(G)$. As $K_n \to K$ in the sense of weak operator convergence on $L^2(G)$, when we pass to the limit in the right-hand side of \eqref{eq:duality_dec} we replace $K_n$ with $K$, thus we conclude that $\langle K_n , \phi \rangle \to \langle K, \phi \rangle$. As $\phi \in \Sz(G)$ was arbitrary, this shows that $K_n \to K$ in $\Sz'(G)$.
\end{proof}

\begin{rem}
Observe that $\Sz(G)$ and $\Sz'(G)$ are not invariant under right translations, nor under the involution $f \mapsto f^*$. Nevertheless $Cv^2(G)$ is.
\end{rem}

We now show that, at least in the smallest dimensional case $n=1$, the correspondence $K \mapsto (M_K(\xi))_\xi$ discussed in Proposition \ref{prp:irrep} can be extended to the case where $K$ is a distribution in $Cv^2(G)$. This result should be compared with, e.g., \cite[Section 4]{Kh} and \cite[Section 5]{ET}.

\begin{lem}\label{lem:gft_cv2}
Assume that $n=1$.
For any $K \in Cv^2(G)$, the partial Fourier transform $\Four$, thought of as a unitary isomorphism
\[
\Four : L^2(G) \to L^2(\RR_\xi^*; L^2(\RR_u)),
\]
intertwines the operator of convolution by $K$ on $L^2(G)$ with the direct integral
\[
\int^\oplus_{\RR^*} M_K(\xi) \,d\xi,
\]
of bounded operators on $L^2(\RR_u)$, whose Schwartz kernels $H_{K}^\xi$ are given, in the sense of distributions, by
\begin{equation}\label{eq:kernel_ft_formula}
H_K^\xi(u,u') = (\Four K)(e^{u'} \xi, u - u').
\end{equation}
Moreover
\begin{equation}\label{eq:transl_kernels}
M_K(e^v \xi) = \tau_{-v} M_K(\xi) \tau_{v}
\end{equation}
for any $v \in \RR$ and $\xi \in \RR^*$, and
\[
\|K\|_{Cv^2(G)} = \sup_{\xi \in \RR^*} \| M_K(\xi) \|_{L^2(\RR) \to L^2(\RR)} = \max_{\pm} \| M_K(\pm 1) \|_{L^2(\RR) \to L^2(\RR)}.
\]
\end{lem}
\begin{proof}
Recall from Lemma \ref{lem:partiallytempered} that any distribution $K \in Cv^2(G)$ is a tempered distribution. As a consequence, we can make sense of the partial Fourier transform $\Four K \in \Sz'(G)$ and of the computations at the beginning of Section \ref{s:multiplier}, which show that indeed $\Four$ intertwines the convolution operator by $K$ with the direct integral of the $M_K(\xi)$. Here the identity
from \eqref{eq:mult_kernel}, expressing the integral kernel of $M_K(\xi)$ for $\xi \in \RR^*$, can be be interpreted in the sense of distribution: indeed, if $\xi \neq 0$, then the change of variables $\RR^2 \ni (u,u') \mapsto (e^{u'} \xi, u-u') \in \RR_+ \xi \times \RR$ is nondegenerate.

As the partial Fourier transform $\Four$ intertwines the convolution operator by $K$ and the direct integral of the $M_K(\xi)$, we clearly have
\[
\|K\|_{Cv^2(G)} = \esssup_{\xi \in \RR^*} \| M_K(\xi) \|_{L^2(\RR) \to L^2(\RR)}.
\]
Now, from the formula for $H_K^\xi$ it is immediately checked that \eqref{eq:transl_kernels} holds. This shows in particular that $M_K(\xi)$ and $M_K(\xi/|\xi|)$ are intertwined by an $L^2$-isometry, so they have the same operator norm, and the required characterisation of the norm $\|K\|_{Cv^2(G)}$ follows.
\end{proof}

As a consequence of the $L^p$-boundedness of the Riesz transforms on $G$ (see Theorem \ref{thm:main} above) and the amenability of the group $G$, via transference
we can deduce
a corresponding result for the Riesz transforms associated with the Schr\"odinger operator $\opH = -\partial_s^2 + e^{2s}$ on $\RR$, stated as Theorem \ref{thm:schroedinger_riesz} above.

\begin{proof}[Proof of Theorem \ref{thm:schroedinger_riesz}]
It will be enough to work with the smallest $ax+b$-group $G = \RR_x \rtimes \RR_u$, i.e., here we take $n=1$.

As before, let $\Four$ denote the partial Fourier transform in the variable $x$ on $\RR_x \times \RR_u$. Then it is immediately checked that, for any $\xi \in \RR^*$,
\[
\Four (X_0 f)(\xi,u) = \partial_u \Four f(\xi,u), \qquad \Four (X_1 f)(\xi,u) = i \xi e^u \Four f(\xi,u),
\]
and therefore also
\[
\Four (\opL f)(\xi,u) = \opH^\xi f(\xi,u),
\]
where $\opH^\xi$ is the Schr\"odinger operator on $L^2(\RR_u)$ given by
\[
\opH^\xi = -\partial_u^2 + \xi^2 e^{2u}.
\]
Consequently, $\Four$ intertwines, in the sense of Lemma \ref{lem:gft_cv2},  the Riesz transforms $\Rz_0$ and $\Rz_1$ with the direct integrals
\[
\int_{\RR^*}^\oplus \Rz_0^\xi \,d\xi , \qquad \int_{\RR^*}^\oplus \Rz_1^\xi \,d\xi,
\]
where
\[
\Rz^\xi_{0} \defeq \partial_u (\opH^\xi)^{-1/2}, \qquad \Rz^\xi_{1} \defeq i \xi e^u (\opH^{\xi})^{-1/2};
\]
note that the latter operators are trivially $L^2(\RR)$-bounded for any $\xi \in \RR^*$, with norm at most $1$. We shall show that the $\Rz_j^\xi$ are actually $L^p(\RR)$-bounded for all $p \in (1,\infty)$; the case $\xi = 1$ gives the desired boundedness result.

As the Riesz transforms $\Rz_j$ are $L^2$-bounded left-invariant operators on $G$,
the aforementioned intertwining property can be written, in the notation of Lemma \ref{lem:gft_cv2}, as
\[
\Rz^\xi_{j} = M_{k_{\Rz_j}}(\xi).
\]
As we know that $k_{\Rz_j} \in Cv^p(G)$ for all $p \in (1,\infty)$ by Theorem \ref{thm:main}, the $L^p$-boundedness of the Schr\"odinger Riesz transforms $\Rz_j^\xi$ would follow if we could apply the transference result in Proposition \ref{prp:irrep}\ref{en:irrep_transf} with $K = k_{\Rz_j}$. This is not directly possible as $k_{\Rz_j}$ is not integrable; to overcome this, we shall show that $k_{\Rz_j}$ is approximable by integrable kernels satisfying analogous $L^p$ bounds.

As discussed, e.g., in \cite[Section 4.2]{MV}, we can subordinate the Riesz transforms to the heat propagator and obtain that
\[
k_{\Rz_j} = \lim_{n \to \infty} K_{j,n}, \qquad K_{j,n} = \int_{2^{-n}}^{2^{n}} X_j h_t \frac{dt}{\sqrt{\pi t}},
\]
where $h_t$ is the heat kernel of $\opL$, as in Section \ref{s:sqrtasymp} above, and the convergence is in the sense of the strong operator topology on $L^2(G)$; so,
from Lemma \ref{lem:partiallytempered} we deduce that 
\begin{equation}\label{eq:kernel_weak_conv}
K_{j,n} \to k_{\Rz_j} \qquad \text{in } \Sz'(G),
\end{equation}
and moreover, by \cite[Proposition 4.1(v)]{MV}, 
\[
\|K_{j,n}\|_{Cv^2(G)} \leq \|k_{\Rz_j}\|_{Cv^2(G)} \leq 1.
\]
Additionally, the proof of \cite[Theorem 1.1(i)]{MV} shows that each $K_{j,n}$ satisfies the ``Calder\'on--Zygmund'' assumptions of \cite[Theorem 2.3]{MV} uniformly in $n$, whence we deduce a uniform bound
\begin{equation}\label{eq:uniform_Cvp_bound}
\sup_{n} \| K_{j,n} \|_{Cv^p(G)} < \infty
\end{equation}
for all $p \in (1,2]$.

Now, clearly $K_{j,n} \in L^1(G)$, as $\sup_{t \in [2^{-n},2^n]} \|X_j h_t\|_1 < \infty$ (see, e.g., \cite[Proposition 3.1]{MV}). Transference (as stated in Proposition \ref{prp:irrep}\ref{en:irrep_transf} above) can then be applied to the $K_{j,n}$, and from \eqref{eq:uniform_Cvp_bound} we deduce that, for any $\xi \in \RR^*$ and $p \in (1,2]$, 
\[
\sup_{n} \|M_{K_{j,n}}(\xi)\|_{L^p(\RR) \to L^p(\RR)} < \infty.
\]
On the other hand, from \eqref{eq:kernel_ft_formula} and \eqref{eq:kernel_weak_conv} it follows that
\[
H^\xi_{K_{j,n}} \to H_{k_{\Rz_j}}^\xi \qquad\text{in } \Df'(\RR^2).
\]
As $H^\xi_{K_{j,n}}$ and $H_{k_{\Rz_j}}^\xi$ are the Schwartz kernels of $M_{K_{j,n}}$ and $\Rz_j^\xi$ respectively, this allows us to deduce that the $\Rz_j^\xi$ are also $L^p(\RR)$-bounded for all $p \in (1,2]$. In order to conclude, we need to show that the adjoint operators $(\Rz_j^\xi)^*$ satisfy analogous bounds.

Let $\zeta \in C^\infty_c(G)$ be any cutoff, and let $p \in (1,2]$. As $Cv^p(G)$ is a module over $C^\infty_c(G)$ (see Lemma \ref{lem:cvp_module_cinfty}), we conclude that the kernels $\zeta K_{j,n}$ are also in $Cv^p(G)$ with a uniform bound in $n$. Additionally, $\zeta K_{j,n} \to \zeta k_{\Rz_j}$ in $\Sz'(G)$, so $H^\xi_{\zeta K_{j,n}} \to H^\xi_{\zeta \Rz_j}$ in $\Df'(\RR^2)$ for any $\xi \in \RR^*$. Arguing as before, by means of transference we can then deduce that $M_{\zeta k_{\Rz_j}}(\xi)$ is $L^p(\RR)$-bounded for all $p \in (1,2]$ and $\xi \in \RR^*$. 

Now, from Proposition \ref{prp:local_riesz_asymp} we deduce that $\zeta k_{\Rz_j+\Rz_j^*} \in L^1(G)$, thus trivially $M_{\zeta k_{\Rz_j+\Rz_j^*}}(\xi)$ is $L^p(\RR)$-bounded for all $p \in [1,\infty]$, and by difference we deduce that $M_{\zeta k_{\Rz_j^*}}(\xi)$ and $M_{\zeta k_{\Rz_j-\Rz_j^*}}(\xi)$ are $L^p(\RR)$-bounded for $p \in (1,2]$.

In order to conclude that $(\Rz_0^\xi)^*$ and $(\Rz_1^\xi)^*$ are $L^p(\RR)$-bounded for $p \in (1,2]$, it is therefore enough to show that $M_{(1-\zeta) k_{\Rz_0-\Rz_0^*}}(\xi)$ and $M_{(1-\zeta) k_{\Rz_1^*}}(\xi)$ are. If we choose the cutoff $\zeta$ appropriately, as in the proof of Corollary \ref{cor:full_reduction}, then
\[
(1-\zeta) k_{\Rz_0-\Rz_0^*} = c (\tilde K_0 + K_0) + r_0, \qquad (1-\zeta) k_{\Rz_1} = c K_1 + r_1
\]
for some constant $c$ and some $r_0,r_1 \in L^1(G)$. Consequently, we are reduced to showing that $M_{\tilde K_0}(\xi), M_{K_0}(\xi), M_{K_1(\xi)}$ are $L^p$-bounded for $p \in (1,2]$. While it would be possible to use transference for this as well, here we can proceed more directly.

Indeed, from the expression for $\tilde K_0$ in Proposition \ref{prp:kernel_reduction}, we have that
\[
H_{\tilde K_0}^\xi(u,u') = (\Four r_0)(e^{u'} \xi) \frac{\chi_{|u-u'| \geq 1}}{u-u'}.
\]
In other words, the operator $M_{\tilde K_0}(\xi)$ on $\RR$ is the composition of the operator of convolution by $u \mapsto \frac{\chi_{|u| \geq 1}}{u}$ and the operator of multiplication by $(\Four r_0)(e^{u} \xi)$. The latter is trivially bounded on any $L^p$ as the function $u \mapsto (\Four r_0)(e^{u} \xi)$ is bounded, while the former is $L^p$-bounded for any $p \in (1,\infty)$ as discussed in the proof of Proposition \ref{prp:tildeK0}.

As for the remaining operators $M_{K_0}(\xi)$ and $M_{K_1}(\xi)$, from Section \ref{s:multiplier} we already know that they are $L^2(w)$-bounded for any $w \in A_2(\RR)$, which implies by extrapolation (see, e.g., \cite[Theorem 9.5.3]{Gr2}) that they are also $L^p$-bounded for any $p \in (1,\infty)$.
\end{proof}

\end{document}